\documentclass[12pt,reqno]{amsart}
\usepackage{amsmath,amsfonts,amsthm,amssymb,color,hyperref,verbatim}
\usepackage{float}
\usepackage[T1]{fontenc}
\usepackage{mathtools}
\usepackage{graphicx}
\usepackage{dsfont}
\usepackage{subfigure} 
\usepackage{enumitem}
\usepackage{graphicx}
\usepackage{fancyvrb}
\usepackage{cite}
\makeatletter
\def\section{\@startsection{section}{1}%
	\z@{.7\linespacing\@plus\linespacing}{.5\linespacing}%
	{\bfseries
		\centering
}}
\def\@secnumfont{\bfseries}
\makeatother
\newcommand{\R}{\mathbb R}
\newcommand{\N}{\mathbb N}

\newcommand{\Z}{\mathbb Z}
\newcommand{\E}{\mathbb E}

\newcommand\bns{\beta_n^s}
\newcommand\bnt{\beta_n^t}
\newcommand\wc{\widetilde{\mathcal{C}}}
\newcommand\wnu{\widetilde{\nu}}
\newcommand\wE{\widetilde E}
\newcommand{\Dim}{\mbox{Dim}_H}
\newcommand{\ep}{\varepsilon}

\newcommand{\lp}{\left(}
\newcommand{\rp}{\right)}

\newtheorem{theorem}{Theorem}[section]
\newtheorem{lemma}[theorem]{Lemma}
\newtheorem{prop}[theorem]{Proposition}

\theoremstyle{definition}
\newtheorem{defn}[theorem]{Definition}
\theoremstyle{remark}

\numberwithin{equation}{section}
\newcommand{\dhaus}[1]{\text{Dim}_{H} \left( #1 \right)}
\newcommand{\proj}{\text{proj}_{\theta}}
\newcommand{\ds}{\displaystyle}
\newcommand{\norm}[1]{\left\lVert#1\right\rVert_2}

\begin{document}
	\title[Potential and projection for macroscopic dimension]{Potential method and projection theorems for Macroscopic Hausdorff Dimension}
	\author[L. Daw]{Lara Daw}
	\author[S. Seuret]{St\'ephane Seuret }
	\address{Lara Daw, University of Luxembourg, Department of Mathematics, Luxembourg}
\address{St\'ephane Seuret, Universit\'e Paris-Est, LAMA (UMR 8050), UPEMLV, UPEC, CNRS, F-94010, Cr\'eteil, France} 	\email{lara.daw@uni.lu, seuret@u-pec.fr }
	\begin{abstract}
		The macroscopic Hausdorff dimension $\dhaus{E}$ of a set $E\subset \R^d$ was introduced by Barlow and Taylor  to quantify a "fractal at large scales" behavior of unbounded, possibly discrete,  sets $E$. We develop a method based on potential theory in order to estimate   this dimension   in $\R^d$. Then, we apply this method to obtain Marstrand-like projection theorems: given a set $E\subset \R^2$, for almost every $\theta\in [0,2\pi]$, the projection of $E$ on the straight line passing through 0 with angle $\theta$ has  dimension  equal to $\min(\dhaus{E},1)$.
		
			\end{abstract}
	
	\maketitle
	
	\medskip\noindent
	
	
	
	\allowdisplaybreaks

%%%%%%%%%%%%%%%%%%%%%%%%%%%%%%%%%%%%%%%%%%%%%%%%%%%%%%%
%%%%%%%%Macroscopic Hausdorff Dimension %%%%%%%%%%%%%%%
%%%%%%%%%%%%%%%%%%%%%%%%%%%%%%%%%%%%%%%%%%%%%%%%%%%%%%%
\section{Introduction}

Fractal geometry provide a general framework for studying   sets possessing either irregular or self-reproducing (deterministic or random, self-similar or self-affine)  properties. Most definitions of fractal dimensions of sets included in  $\R^d$ are based on the  local properties (also known as microscopic) of the set. Taking into consideration that many statistical physics models are built on discrete spaces, Barlow and Taylor \cite{barlow1992, barlow1989} introduced  a new notion of dimension to study unbounded "fractal-like" sets   on discrete space. This  so-called macroscopic Hausdorff dimension (see Definition \ref{defMacro}  below) has proved to be useful in  quantifying the behavior at infinity of several objects, beyond   the transient range of random walks in $\Z^d$  which was the original motivation of Barlow and Taylor in \cite{barlow1992}. 

Macroscopic Hausdorff dimension is actually defined for every set (not only discrete) in $\R^d$ \cite{barlow1992}. It is a discrete analog of Hausdorff dimension, and the word macroscopic comes from the fact that this dimension ignores the local structure of the sets. At the same time, the macroscopic Hausdorff dimension assesses the asymptotic behavior at infinity of the sets, so it is very relevant when one is interested in the description of infinite objects, how they fill the space "at large scale". The macroscopic Hausdorff dimension was a key tool used by  Xiao et Zheng \cite{xiao2013} in studying the range of a random walk in random environment. It is   related to \cite{khoshnevisan2017macroscopic} where Khoshnevisan and Xiao are concerned with the macroscopic   geometry of other  random sets. In \cite{khoshnevisan2017intermittency}, Khoshnevisan, Kim and Xiao found out a multifractal behavior for the macroscopic dimension of tall peaks of solutions to stochastic PDEs.   Georgiou \emph{et el} \cite{georgiou2018dimension} solved Barlow and Taylor question \cite[Problem, p. 145]{barlow1992} by qualifying the range of
an arbitrary transient random walk. The macroscopic Hausdorff dimension was also useful for studying the large scale structure of sojourn sets associated to the Brownian motion \cite{seuret2019sojourn}, the fractional Brownian motion \cite{nourdin2018sojourn, daw2021uniform}, and the Rosenblatt process \cite{daw2021fractal}. 

In this paper we are interested in building various methods for estimating the macroscopic Hausdorff dimension. Recalling the fact that macroscopic Hausdorff dimension is a discrete analog of the Hausdorff dimension, we start by stating the estimating methods used for the Hausdorff dimension. In most cases, when estimating  the Hausdorff dimension of a set $E$, the difficult part consists in finding a suitable lower bound for $\dim_H(E)$.
Various methods exist to find lower bounds for the standard Hausdorff dimension, and it is a natural question to ask whether these methods have their counterparts for the macroscopic Hausdorff dimension. The two usual techniques are the mass distribution principle and the potential theoretic method.

The mass distribution principle, see  for instance \cite[page 67]{falconer2004fractal}, states that if a set $E\subset \R^d$ and a Borel finite measure $\mu$ are such that $\mu(E)=1$ and $\mu(B(x,r))\leq C r^s$ for every $x\in \R^d$ and $r>0$, then the $s$-dimensional Hausdorff measure $\mathcal{H}^s (E)$  is larger than $\mu(E)/C$, and so $E$ has at least Hausdorff dimension $s$. 

The potential theoretic method is based on an integral analysis: if for some probability measure $\mu$, $\mu(E)=1$ and  the integral $\displaystyle \iint_{(\R^d)^2} \frac{d\mu(x)d\mu(y)}{\norm{x-y}^s}$ is finite, then  again $ E$ has at least Hausdorff dimension $s$. In addition to bounding the Hausdorff dimension from below, the potential theoretic method plays a key role in proving the projection theorem.

The first aim of this paper is to establish similar results for the macroscopic Hausdorff dimension. This happens to be very easy for the mass distribution principle, and follows essentially from previous works. It is much more challenging for  the potential theoretic method, and a careful analysis is needed.

As an application of the new potential theoretic method, we   obtain a Marstrand-like projection theorem, describing the dimension of almost all projections on lines of sets $E\in \R^2$. 
Dealing with the dimensions of projections of Borel sets is a line of research that has a long history. It started with the investigation  by Marstrand \cite{marstrand1954some} of the projection theorem associated to the Hausdorff dimension. He dealt with orthogonal projections on linear subspaces and  proved that 
\begin{align*}
\mbox{for every Borel set $E \subset \R^2$, } \ \     \mbox{dim}_H (\mbox{proj}_V E) =  \min \{ \mbox{dim}_H E,1\}
\end{align*}
for almost every 1-dimensional subspaces $V$, where $\mbox{proj}_V$ denotes the orthogonal projection onto $V$ and $\mbox{dim}_H E$ denotes the Hausdorff dimension of $E$. Afterwards Marstrand's results was proved by Kaufman but using potential theoretic methods \cite{kaufman1968hausdorff}. Subsequently in 1975 Mattila extended these results to Borel sets $E \subset \R^n$ and almost all $V$ in the Grassmannian $G(n, m) $\cite{mattila1975hausdorff}.   We prove analog results for the macroscopic Hausdorff dimension, using the potential theory method we developed above.

\section{Definitions and statements of the results}

Here and in the reset of the paper, let $(\R^d, \norm{.})$ be the  $d$-dimensional Euclidean space equipped with the  $L^2$- norm.

\subsection{The macroscopic Hausdorff dimension}

For $x \in \R^d$ and $r>0$, $B(x,r)$ denotes the Euclidean  ball with  center $x$ and radius $r$.
For $E \subset \R^d$, the diameter of a set $E$ is   denoted by $|E|$.
%
%\begin{align*}
%	d(E) = \min \left\{r \, : \, E \subseteq B(x,r) \mbox{ for some } x \in \mathbb{R}^d\right\}
%\end{align*}

Let us recall the definition of the  Barlow-Taylor macroscopic Hausdorff dimension $\dhaus{E}$ of a set $E \subseteq \R^d$, developed in \cite{barlow1989, barlow1992}.

Define, for all integer $n \in \N$, the $n$-th shell of $\R^d$ by
\begin{align}
	\label{S_n}
 S_0= B(0,1) \quad \mbox{and} \quad  S_{n} := B(0,2^{n}) \setminus   B(0,2^{n-1}) \mbox{ for all } n \geq 1.
\end{align}

Like   the standard Hausdorff dimension, the macroscopic Hausdorff dimension $\Dim(E)$ aims at describing how a set $E$ can be  efficiently covered by balls. Since $\Dim$ is concerned only with large scale behaviors, Barlow and Taylor proposed to study the covers of the intersections $E\cap S_n$ by balls,  for every $n\in \N$, and the balls used to cover the sets $E\cap S_n$ will all be of diameter at least 1. Again this is justified by the fact that   this dimension is supposed to  describe discrete sets (so small balls are not relevant).

To this end, let us introduce, for  $E \subseteq \R^d$,  the set of {\it  covers} of $E$ restricted to $S_n$ defined by 
%%%%%%%%%%%%%%%%%%%%%%%%%%%%%%%%%%%%%%%%%%%%%%%%%%%%%%%
\begin{align*}
	\wc_n(E) = \left\{ 
	\begin{matrix}
		\left\{B(x_i,r_i)\right\}_{i=1}^{m}\, : & m\in \N,  \, x_i \in S_n,  \, r_i \geq 1,  \,   E \cap S_n \subset \bigcup_{i=1}^{m} B(x_i,r_i)
	\end{matrix} \right\}.
\end{align*}

Finally, for $s \geq 0$ and $n \in \N$, set
%%%%%%%%%%%%%%%%%%%%%%%%%%%%%%%%%%%%%%%%%%%%%%%%%%%%%%%
\begin{align}
	\label{def:nu^n_s_0}
	\wnu^{s}_{n}(E)  & = \inf \left\{ \sum_{i=1}^{m} \left(\dfrac{r_i}{2^{n} }\right )^s : \, \left\{ B_i=B(x_i,r_i)\right\}_{i=1}^{m} \in \wc_n(E)\right \}.
\end{align}

Observe that $\wnu_{n}^{s}$ is sub-additive, i.e. $\wnu^{s}_{n}(A\cup B)\leq \wnu_{n}^{s}(A)+\wnu_{n}^{s}(B)$ for every sets $A$ and $B$, but is not a measure (because of the constraints on $r_i$).

%--------------------- DEFINITION --------------------%
\begin{defn}
When $\wnu_{n}^{s}(E)= \sum_{i=1}^{m} \left(\dfrac{r_i}{2^{n} }\right )^s $ and $E \cap S_n \subset \bigcup_{i=1}^{m} B(x_i,r_i)
$,   the finite family of balls $\left\{ B_i=B(x_i,r_i)\right\}_{i=1}^{m}$ is called an $s$-optimal cover of $E \cap S_n$.
\end{defn}
%--------------------- DEFINITION --------------------%

The existence of optimal covers is not guaranteed. We will deal with this issue in Section \ref{sec:prop_MHD}.

We are now ready to define the Barlow-Taylor macroscopic Hausdorff dimension.
%--------------------- DEFINITION --------------------%
\begin{defn}
\label{defMacro} 
For every $s\geq 0$ and $E\subset \R^d$, define 
	\begin{align*}
	\wnu^s(E)  & = \sum_{n\geq 1} \wnu_n^s(E).
\end{align*}
The macroscopic Hausdorff dimension of $E \subset \mathbb{R}^d$ is defined by 
	\begin{align}
		\dhaus{E} = \inf  \left\{ s \geq0 : \,\wnu^s(E)  < + \infty \right \}.
		\label{DefDim0}
	\end{align}
\end{defn}
%--------------------- DEFINITION --------------------%

One easily checks that  $\dhaus{E} \in [0,d]$ for all  $E \subset \mathbb{R}^d$, that $\dhaus{E} =0$ when $E$ is bounded,  and that  an alternative definition for $\dhaus{E} $ is
\begin{align*}
		\dhaus{E} = \sup  \left\{ s \geq 0 : \, \wnu^s(E) = + \infty \right \},
	\end{align*}
where $\sup \emptyset =0$ by convention. It is also standard that $\dhaus{f(E)}\leq \dhaus{E}$ for every Lipschitz mapping $f:\R^d\to\R^d$.

A key ingredient  when working with  the standard Hausdorff dimension  is the existence of $s$-sets, i.e. sets $E\subset \R^d$  with Hausdorff dimension $\dim_H(E)=s$ and such that its $s$-Hausdorff measure $\mathcal{H}^s(E)$ is finite. We introduce a similar notion for the macroscopic Hausdorff dimension.

%----------------------DEFINITION--------------------%
\begin{defn}
Let $s\geq 0$.	A set  $E \subset \R^d$ is called   a macroscopic $s$-set when $\dhaus{E} = s$ and $\ds  \wnu^s(E) < + \infty$.
\end{defn}
%----------------------DEFINITION--------------------%

We prove the existence of macroscopic $s$-sets.

%-----------------------THEOREM----------------------%
\begin{theorem}
	\label{thm:s-set}
	Let $E \subset \R^d$ be such that $\ds \wnu^s(E) = + \infty$. Then there exists  a macroscopic $s$-set $\widetilde{E}$ such that $\widetilde{E} \subset E$.
	 \end{theorem}

	  This extraction theorem  is  a key ingredient at various places in our proofs.

\subsection{Methods to find lower bounds for $\dhaus{E}$}

For every set $B$ and every measure $\mu$,  $\mu_{|B} $ stands for the restriction of $\mu $ on $B$, i.e. $\mu_{|B}(A)=\mu(A\cap B)$.

As recalled above, the mass distribution principle is a powerful, albeit simple, tool allowing to find  a lower bound of the Hausdorff dimension by considering measures supported on the set, see \cite[page 67]{falconer2004fractal}. We prove a similar result for the macroscopic Hausdorff dimension $\Dim$.
\begin{prop}[Macroscopic mass distribution principle]
	\label{prop:mass_dist}
	Let $E$ be a Borel subset of $\R^d$ and $s>0$. Suppose  that there exists a  Radon  measure $\mu$ on $\R^d$ such that  $ \mu(E) = +\infty$ and a constant $c>0$ such that for all $n \in \N $ , $x \in S_n$ and $1 \leq r \leq 2^n$,
	\begin{align*}
		\mu_{|S_n}(B(x,r)) \leq c \left(\frac{r}{2^{n}}\right)^s.
	\end{align*}
Then, for all $n \in \N $, $\wnu^s_n(E) \geq \dfrac{\mu_{|S_n}(E )}{c}$ and $\dhaus{E} \geq s$.
\end{prop}

 The  proof of the macroscopic mass distribution principle is not complicated. Although it was not exactly stated before as we write it,  it  essentially follows  directly from previous results, and so it  is  not so innovative.
 
 \smallskip

 This is not the case for the potential method below. Let us first introduce the macroscopic  $s$-energy of a measure. 
 %---------------------DEFINITION---------------------%
\begin{defn}
	Let  $s \geq 0$, and let  $\mu$ be a finite mass distribution on $\R^d$. The macroscopic $(\mu,s)$-potential at a point $x$ is defined as 
	\begin{align}
	\label{defphi}
		\phi^s_\mu(x) := \int_{\R^d} \dfrac{d\mu(y)}{\norm{x-y}^s \vee 1 }.
	\end{align}
	The macroscopic $s$-energy of $\mu$ is 
	\begin{align}
	\label{defIs}
		I_s(\mu):= \int_{\R^d} \phi^s_\mu(x)d\mu(x)=  \iint_{(\R^d)^2}\dfrac{d\mu(x)d\mu(y)}{\norm{x-y}^s \vee 1 }.
	\end{align}
\end{defn}
%---------------------DEFINITION---------------------%
In the case of standard Hausdorff dimension, in the integrals \eqref{defphi} and \eqref{defIs},  the quantity $\norm{x-y}^s \vee 1$  is simply  $\norm{x-y}^s  $. This modification is justified by the fact that   $\Dim$ is not concerned with local behavior, so we are not interested in small interactions $\norm{x-y}<1$.

%----------------------THEOREM----------------------%
\begin{theorem}
	\label{thm:potential}
	Let $E$ be a subset of $\R^d$.  
	\smallskip
		\begin{enumerate}
		%-----------------------------------------------%
	
	\item If there exists a Radon measure $\mu $ on $\R^d$  such that $\mu(E)=+\infty$ and if
		$$\ds \sum_{n \geq 0} 2^{ns} I_s(\mu_{|S_n}) < +\infty,$$
		 then   $\wnu^s(E)=+\infty$ and   $\dhaus{E}\geq s$.
		%-----------------------------------------------%
	\smallskip
		\item If $  \wnu^s(E) = + \infty$, then for all $0<\ep<s$ there exists a Radon  measure $\mu^{\ep}$ on $\R^d$ such that  $ \mu^{\ep}(E )= +\infty$ and  $\ds \sum_{n\geq 0} 2^{n(s - \ep)} I_{s -\ep}(\mu_{|S_n}^{\ep}) < +\infty$.
		%-----------------------------------------------%
	\end{enumerate}
\end{theorem}
%----------------------THEOREM----------------------%
The potential theoretic methods we demonstrated in Theorem \ref{thm:potential} are very comparable to the ones established for the standard Hausdorff dimension \cite[Theorem 4.13]{falconer2003}. Unlike the standard Hausdorff dimension case, for the macroscopic Hausdorff dimension, we consider the measure $\mu$ is define on $\R^d$, and we focus on the restriction of $\mu$ on every annulus $S_n$. For this reason, we deal with sums over $n$.

\subsection{Application to projections}

Projection theorems for Hausdorff dimensions  have  recently  regained a lot of attention  after some breakthroughs by M. Hochman  and P. Shmerkin \cite{hochman2012local} and others, who used these theorems to tackle many longstanding questions in geometric measure theory and dynamical systems. It is quite satisfactory that they have natural counterparts  in terms of macroscopic Hausdorff dimensions, as stated in the following theorem.

%----------------------THEOREM----------------------%
\begin{theorem}
\label{thm:projection}
	Let $E \subset \R^2$ be a Borel set. Define   $L_\theta$  as the straight line passing through 0 with angle $\theta$, and $\proj E$ as the orthogonal projection of $E$ onto $L_\theta$.
	\begin{itemize}
		\item[(a)] If $\dhaus{E} < 1$, then $\dhaus{\proj E} = \dhaus{E}$ for Lebesgue almost every $\theta \in [0,\pi]$.
		\item[(b)] If $\dhaus{E} \geq 1$, then $\dhaus{\proj E} = 1$ for Lebesgue almost every $\theta \in [0,\pi]$.
	\end{itemize}
\end{theorem}
%----------------------THEOREM----------------------%

As in the standard Hausdorff dimension case, the proof is based on a subtle use of the potential method and Theorem \ref{thm:potential}.

It can be expected that Theorem \ref{thm:projection} can be extended in higher dimensional spaces, and that both Theorem \ref{thm:potential} and Theorem \ref{thm:projection}  are useful in other situations that the one we describe here.

The structure of  the paper is as follows. The main three results, Theorems \ref{thm:s-set}, \ref{thm:potential} and \ref{thm:projection} are established in Sections \ref{sec:extraction}, \ref{sec:potentialMethods}, and \ref{sec:projection} respectively. Some necessary technical properties of the macroscopic Hausdorff dimension are proved in Section \ref{sec:prop_MHD}.  

%%%%%%%%%%%%%%%%%%%%%%%%%%%%%%%%%%%%%%%%%%%%%%%%%%%%%%%
%%%% Properties of Macroscopic Hausdorff Dimension %%%%
%%%%%%%%%%%%%%%%%%%%%%%%%%%%%%%%%%%%%%%%%%%%%%%%%%%%%%%
\section{First properties of Macroscopic Hausdorff Dimension}
\label{sec:prop_MHD}

\subsection{An alternative definition for the macroscopic Hausdorff dimension}

We will use an alternative,  easier to handle with, definition for the macroscopic Hausdorff dimension, based on a simple modification of the $\wnu_n^s$ quantities. We     restrict ourselves to covers centered on integer points, with integer radii. We show that, up to a constants, this does not modify the values of the quantities involved in the computations, and   the value of the macroscopic Hausdorff dimension is left unchanged.

We introduce  for  $E \subseteq \R^d$ and $n\geq 0$,  the set of {\it  proper covers} of $E$ restricted to $S_n$ by 
%%%%%%%%%%%%%%%%%%%%%%%%%%%%%%%%%%%%%%%%%%%%%%%%%%%%%%%
\begin{align*}
	\mathcal{C}_n(E) = \left\{ 
	\begin{matrix}
		\left\{B(x_i,r_i)\right\}_{i=1}^{m}\, : & m\in \N,  \, x_i \in \Z^d\cap S_n,  \, r_i \in \N^*,  \,   E \cap S_n \subset \bigcup_{i=1}^{m} B(x_i,r_i)
	\end{matrix} \right\}.
\end{align*}

%%%%%%%%%%%%%%%%%%%%%%%%%%%%%%%%%%%%%%%%%%%%%%%%%%%%%%%

\begin{defn}
For every $s\geq 0$, $n\geq 0$  and $E\subset \R^d$, define 
\begin{align}
	\label{def:nu^n_s}
	\nu_{n}^{s}(E)  & = \inf \left\{ \sum_{i=1}^{m} \left(\dfrac{r_i}{2^{n} }\right )^s : \, \left\{ B_i=B(x_i,r_i)\right\}_{i=1}^{m} \in \mathcal{C}_n(E)\right \}
\end{align}
and 	\begin{align}
	\nu^s(E)  & = \sum_{n\geq 1} \nu_n^s(E).
\end{align}
\end{defn}
Due to the fact that the $x_i$ are (multi)-integers, as well as the $r_i$, the above infimum \eqref{def:nu^n_s} in  $\nu_{n}^{s}(E) $ is reached for some cover $ \left\{ B_i=B(x_i,r_i)\right\}_{i=1}^{m} \in \mathcal{C}_n(E)$.

Observe that $\nu_{n}^{s}$ is still sub-additive, i.e. $\nu_{n}^{s}(A\cup B)\leq \nu_{n}^{s}(A)+\nu_{n}^{s}(B)$ for every sets $A$ and $B$.

\begin{lemma}
\label{lemma:nu_ndiscete}
For every $n\geq 0$, every set $E\subset \R^d$, one has
\begin{equation}
\label{equivnu}
\wnu_n^s(E) \leq \nu_n^s(E)  \leq  (2+\sqrt{d})^s \wnu_n^s(E)  .
\end{equation}
In particular, one still has
	\begin{align}
		\dhaus{E} = \inf  \left\{ s \geq0 : \,\nu^s(E)  < + \infty \right \} =  \sup  \left\{ s \geq 0 : \, \nu^s(E) = + \infty \right \} .
		\label{DefDim}
	\end{align}
\end{lemma}
%--------------------- DEFINITION --------------------%

%%%%%%%%%%%%%%%%%%%%%%%%%%
\begin{proof}

The fact that  $\mathcal{C}_n(E)\subset \wc_n(E)$ implies directly that $\wnu_n^s(E) \leq \nu_n^s(E) $.

Now, let $\left\{B(\tilde x_i, \tilde r_i)\right\}_{i=1}^{m}\ \in \wc_n(E)$. Each ball $B(\tilde x_i, \tilde r_i) $ is included in a ball $B(  x_i, \tilde r_i+\sqrt{d})$, where $x_i\in \Z^d\cap E_n$. So $\left\{B \lp x_i, \left\lceil \tilde r_i+\sqrt{d} \right\rceil \rp\right\}_{i=1}^{m}\ \in \mathcal{C}_n(E)$, and  using that  $\left\lceil \tilde r_i+\sqrt{d} \right\rceil  \leq \tilde r_i +\sqrt{d} +1 \leq (2+\sqrt{d})\tilde r_i $ (since $\tilde r_i \geq 1$), one has
$$\sum_{i=1}^{m} \left(\dfrac{\tilde r_i+\sqrt{d}} {2^{n} }\right )^s \leq (2+\sqrt{d})^s \sum_{i=1}^{m} \left(\dfrac{\tilde r_i}{2^{n} }\right )^s.$$
This holds for any cover $\left\{B(\tilde x_i, \tilde r_i)\right\}_{i=1}^{m}\ \in \wc_n(E)$, so $\nu_n^s(E) \leq (2+\sqrt{d})^s  \wnu_n^s(E)$.
\end{proof} 
%%%%%%%%%%%%%%%%%%%%%%%%%%

Lemma \ref{lemma:nu_ndiscete} shows in particular that the convergence/divergence properties of $\wnu^s(E)$ and $\nu^s(E)$ are identical.

The main advantage of dealing with $\nu^s(E)$  is   the existence of optimal proper $s$-covers, i.e. covers $ \left\{ B_i=B(x_i,r_i)\right\}_{i=1}^{m} \in \mathcal{C}_n(E)$ such that $ \nu_{n}^{s}(E)  = \sum_{i=1}^{m} \left(\dfrac{r_i}{2^{n} }\right )^s$. These optimal covers exists because $x_i$ and $r_i$ are positive integers.

\medskip

 In our further analysis, the size of  the balls of   optimal covers will matter, justifying the following definition. 
\begin{defn}
For $E \subset \Z^d$, $n \in \N$ and $0<s<d$, define
    \label{eq:b_s^n}
    \begin{align*}
	\bns(E) :=  \max \left\{ \max_{1\leq i\leq p} \dfrac{r_i}{2^n} \, : \, \left(B(x_i,r_i)\right)_{i=1}^{p} \mbox{ is an $s$-optimal proper cover of } E \cap S_n  \right\}.
\end{align*}
\end{defn}

The quantity $\bns(E) $ will be important, in particular for Theorem \ref{thm:potential} about potential methods and for the projection Theorem \ref{thm:projection}.

%%%%%%%%%%%%%%%%%%%%%%%%%%
%%%%%%%%%%%%%%%%%%%%%%%%%%
%%%%%%%%%%%%%%%%%%%%%%%%%%
%%%%%%%%%%%%%%%%%%%%%%%%%%
%%%%%%%%%%%%%%%%%%%%%%%%%%
\subsection{Some preliminary results}

We first prove two propositions that will be needed later.
%---------------------PROPOSITION---------------------%
\begin{prop}
	\label{prop_mu_n}
	Let $\mu_n$ be a Borel measure on $S_n$, $E \subset \R^d$ be a Borel set and $0<c<+\infty$ be a constant.
	\begin{enumerate}
		\item[a)] If $\ds \max_{ r \in \N^*} \dfrac{\mu_n \left(B(x,r)\right)}{\left(r/2^n\right)^s} \leq c$ for all $x \in   E \cap S_n$, then $\ds \nu^s_n(E) \geq \dfrac{\mu_n(E)}{c2^s}$.
		\item[b)] If $\ds \max_{ r  \in \N^*} \dfrac{\mu_n \left(B(x,r)\right)}{\left(r/2^n\right)^s} > c$ for all $x \in   E \cap S_n$, then $\ds \nu^s_n(E) \leq \frac{(5(1+\sqrt{d}/2))^s }{c}   \mu_n(S_n)$.
	\end{enumerate}
\end{prop}
%---------------------PROPOSITION---------------------%
%----------------------PROOF----------------------%
\begin{proof}
a)  Let $\left\{B(x_i,r_i)\right\}_{i=1}^m \in \mathcal{C}_n(E)$. For each $1 \leq i \leq m$, there exists $y_i \in B(x_i,r_i) \cap E \cap S_n$ such that $B(x_i,r_i) \subset B(y_i,2r_i)$, so
		\begin{align*}
		    \mu_n(B(x_i,r_i)) \leq \mu_n(B(y_i,2r_i)) \leq c \left(\dfrac{2r_i}{2^n}\right)^s = c2^s\left(\dfrac{r_i}{2^n}\right)^s.
		\end{align*}
		 Then,
		 \begin{align*}
		     \mu_n(E \cap S_n) \leq \sum_{i=1}^m \mu_n(B(x_i,r_i)) \leq c 2^s \sum_{i=1}^m \left(\dfrac{r_i}{2^n}\right)^s,
		 \end{align*}
		  which is true for all covers $\left\{B(x_i,r_i)\right\}_{i=1}^m \in \mathcal{C}_n(E)$. Finally, taking the infimum over all elements of $ \mathcal{C}_n(E)$, one gets
		\begin{align*}
			\mu_n(E)= \mu_n(E \cap S_n) \leq c 2^s \nu^s_n(E).
		\end{align*}

b) Consider the   family of balls   
		\begin{align*}
			\mathcal{B}_n \:= \left\{B(x,r) : \, x \in   E \cap S_n, \, r\in \{1,2,..., 2^n\} \mbox{ and } \mu_n(B(x,r)) > c\left(\dfrac{r}{2^n}\right)^s\right\}  .
		\end{align*}
		Then
		\begin{align*}
		    \ds E \cap S_n \subset \bigcup_{B(x,r) \in \mathcal{B}_n} B(x,r).
		\end{align*}
		Now, we invoke the following $5r$-covering Lemma \cite[Lemma 4.8]{evans2015measure}.
				\begin{lemma}
		    \label{lemma:covers}
			Let $\mathcal{B}$ be a family of balls in $\R^N$ and suppose that $
				\sup_{B \in \mathcal{B}} d(B) < \infty$.			Then there exists a countable sub-family of disjoint balls $\mathcal{B}_0$ of $\mathcal{B}$   such that 
			\begin{align*}
				\bigcup_{B \in \mathcal{B}} B \subset \bigcup_{i \in \mathcal{B}_0} 5B_i.
			\end{align*}			
		\end{lemma}
		Using the previous lemma, there exists a finite family  $(B_i= B(x_i,r_i))_{i=1,...,m}$   of disjoint balls, all elements of  $\mathcal{B}_n $, such that $ 			\bigcup_{B \in  { \mathcal{B}_n}} B     \subset \bigcup_{i=1}^m 5B_i$. 		The finiteness of the family comes from the boundedness of $S_n$ and the fact that the balls all have a diameter greater than 1. Up to a small translation of each $x_i$ by a vector of length at most $\sqrt{d}/2$, one can assume that $x_i \in \Z^d$ and that \begin{align*}
			\bigcup_{B \in  { \mathcal{B}_n}} B \subset \bigcup_{i=1}^m 5 B \lp  x_i, \left \lceil r_i+\sqrt{d}/2 \right \rceil \rp .
		\end{align*} 
		With the translations that we added, some balls $B \in  { \mathcal{B}_n}$ may intersect, but this does not affect our argument.
		
		Using the definition of $\nu_n^s(E)$, one finally gets
		\begin{align*}
			\nu^s_n(E) &  \leq \sum_{i=1}^m \left(\dfrac{5\left \lceil r_i+\sqrt{d}/2 \right \rceil }{2^n}\right)^s \leq (5(2+\sqrt{d}/2))^s  \sum_{i=1}^m \left(\dfrac{r_i}{2^n}\right)^s  \\
			& \leq \frac{(5(2+\sqrt{d}/2))^s}{c}  \sum_{i=1}^m \mu_n(B_i) \leq \frac{(5(2+\sqrt{d}/2))^s}{c}  \mu_n(S_n),
		\end{align*} 
where the last equality comes from the disjointness of the  $B_i$s.
		\end{proof}
%----------------------PROOF----------------------%
The following proposition guarantees that given a measure $\mu$ on a set $E$, there exists a smaller set $F\subset E$ such that the measure $\mu$ has a controlled   local scaling behavior on $F$.
%---------------------PROPOSITION---------------------%
\begin{prop}
	\label{propab}
	Let $E \subset \R^d$ be a Borel set. Then, for every $0<s \leq d$ there exists a constant $c_s>0$ (depending only on $s$) and a set $\emptyset \neq F \subset E$ such that  for every $n\geq 1$, 
	\begin{itemize}
		\item[(a)] $\dfrac{4}{5} \nu^{s}_n(E) \leq \nu^{s}_n(F) \leq \nu^{s}_n(E)$
		\item[(b)] $\nu^s_n\left( F \cap B\left(x,r\right) \right) \leq c_s \left(\dfrac{r}{2^n}\right)^s$ for all $x \in \Z^d\cap E_n$ and $r \geq1$.
	\end{itemize}

\end{prop}
%---------------------PROPOSITION---------------------%
%----------------------PROOF----------------------%
\begin{proof}
	Let $E \subset \R^d$ and set for every $n\geq 1$ 
	\begin{align*}
		F_n := \left\{x \in E\cap S_n : \, \max_{ r \geq 1} \dfrac{\nu^s_n\left(E \cap B(x,r)\right) }{\left(r/2^n\right)^s} > 5(5(2+\sqrt{d}/2))^s\right\}	.
	\end{align*}
	Using Proposition \ref{prop_mu_n} (b) applied to the set $F_n$ and the measure  $\mu_n(A)= \nu_n^s(E \cap A)$, one gets 
	\begin{align*}
	    \mu_n(F_n) \leq (5(2+\sqrt{d}/2))^s 5^{-1} (2+\sqrt{d}/2))^{-s} \mu_n(S_n ) = \frac{1}{5} \mu_n(E).
	\end{align*}
	Then $\mu_n(E \textbackslash F_n) \geq \frac{4}{5} \mu_n(E)$, i.e. as soon as $E\cap S_n$ is not empty, $\ds (E \textbackslash F_n)\cap S_n \neq \emptyset$. Finally, the set $\ds F = \bigcup_{n \geq 0} E \textbackslash F_n$ satisfies the two conditions mentioned above, with the constant $c_s= 5(5(2+\sqrt{d}/2))^s$.
\end{proof}
%----------------------PROOF----------------------%

%-----------------------------------------------------%
%-----Mass distribution principle for $\dhaus{.}$-----%
%-----------------------------------------------------%
\subsection{Proof of the mass distribution principle : Proposition \ref{prop:mass_dist}}
\label{subsec:mass_dist}
%---------------------PROPOSITION---------------------%
 
	For $n \in \N$, let $\left\{B(x_i,r_i)\right\}_{i=1}^m \in \wc_n(E)$, then 
	\begin{align*}
		\mu_{|S_n}(E \cap S_n) \leq \mu_{|S_n} \left(\bigcup_{i=1}^m B(x_i,r_i)\right) \leq \sum_{i=1}^{m} \mu_{|S_n}(B(x_i,r_i)) \leq c\sum_{i=1}^{m} \left(\frac{r_i}{2^{n}}\right)^s.
	\end{align*}
	Taking infimum over all proper covers $\left\{B(x_i,r_i)\right\}_{i=1}^m \in \wc_n(E)$, one gets 
	\begin{align*}
	  \dfrac{\mu_{|S_n}(E \cap S_n)}{c} \leq \nu_{n}^s(E)  .
	\end{align*}
	Then $\ds \wnu^s(E)  \geq \frac{\sum_{n\geq 0} \mu_{|S_n}(E)}{c}=  \frac{\mu(E)}{c}= +\infty$ and so $\dhaus{E} \geq s$.

\medskip

Observe that the same proof works if  $\wc_n(E)$ and $\wnu_n^s(E)$ are replaced respectively by $\mathcal{C}_n(E)$ and $\nu_n^s(E)$.

%-----------------------------------------------------%
%-----Subsets of macroscopic finite measure-----%
%-----------------------------------------------------%

%%%%%%%%%%%%%%%%%%%%%%%%%%%%%%%%%%%%%%%%%%%%%%%%%%%%%%%
%%%% Properties of Macroscopic Hausdorff Dimension %%%%
%%%%%%%%%%%%%%%%%%%%%%%%%%%%%%%%%%%%%%%%%%%%%%%%%%%%%%%
\section{Subsets of finite macroscopic measure} 
\label{sec:extraction}

In this section,  we   prove a stronger version than Theorem \ref{thm:s-set}, more precisely:
\begin{theorem}
	\label{thm:s-set-2}
	Let $E \subset \R^d$ such that $\ds \nu^s(E) = + \infty$. Then there exists  a macroscopic $s$-set $\widetilde{E}$ such that $\widetilde{E} \subset E$ and $\lim_{n\to +\infty} \sup_{t\in [0,d]} \bnt(\wE)=0$.
	 \end{theorem}

Observe that we can either work with $\wnu^s$ or $\nu^s$, since  $(\wnu^s(E)<+\infty) \Leftrightarrow (\nu^s(E)<+\infty)$. We choose to work with $\nu^s$, and in this case $\bns(\wE)$ is defined without ambiguity.

We start with three technical lemmas, that will later help us extract a macroscopic $s$-set and prove the projection theorem.

 	\begin{lemma}
		\label{lemma:convSum1}
		Let $(a_n)_{n \geq 1}$ be a bounded sequence of positive real numbers, such that  $\ds \lim_{n\to +\infty} A_n:=\sum_{k=1}^{+\infty} a_k  =+\infty$. For every $\ep>0$, $\ds \sum_{n=1}^{+\infty}  \dfrac{a_n}{A_n^{1+\ep}}<+\infty $  and $\ds \sum_{n=1}^{+\infty}  \dfrac{u_n}{A_n}=+\infty $.	\end{lemma}
		
This is a standard exercise, we prove it for completness.		
	\begin{proof}
		Let $\ep>0$. For $n\geq 2$ and $\ep>0$, one has  ]$ \ds 
			\int_{A_{n-1}}^{A_n} \dfrac{dx}{x^{1 + \ep}} \geq \int_{A_{n-1}}^{A_n} \dfrac{dx}{A_n^{1 + \ep}} = \dfrac{u_n}{A_n^{1 + \ep}}.$
		Then,  $ \ds
			\dfrac{1}{\ep}\dfrac{1}{A_1^{\ep}} \geq \dfrac{1}{\ep} \left(\dfrac{1}{A_1^{\ep}}-\dfrac{1}{A_n^{\ep}}\right) = 	\int_{A_{1}}^{A_n} \dfrac{dx}{x^{1+\ep}} \geq \sum_{k=2}^{n}\dfrac{a_k}{A_k^{1+\ep}}.
	$
So the sums $\ds \sum_{k=1}^n  \frac{a_n}{A_n^{1+\ep}}$ are uniformly bounded and the series converges. 
Similarly, $ \ds \ln(A_n) - \ln(A_1)=\int_{A_n}^{A_1}\dfrac{dx}{x}   \leq \sum_{k=2}^{n}\dfrac{a_k}{A_{k-1}}.$ Since $A_n \rightarrow + \infty$ as $n \rightarrow + \infty$, the series $ \ds\sum_{k=2}^{n}\dfrac{a_k}{A_{k-1}}$ diverges. Also, since  $(a_n)$ is bounded, $A_n\sim A_{n-1}$ and the   series $\ds \sum_{k=2}^{n}\dfrac{a_k}{A_{k}}$ diverges.  	\end{proof}

	\begin{lemma}
		\label{lemma:convSum2}
		Let $(a_n)_{n \geq 1}$  be a positive sequence converging to zero, $(b_n)_{n \geq 1}$ be a bounded sequence of positive real numbers, such that $\sum_{n\geq 1} a_nb_n=+\infty$. Then, there exists a sequence $(c_n)_{n\geq 1}$ such that:
		\begin{enumerate}
		\item either $c_n =b_n$, or $c_n=0$,
		\item $\sum_{n\geq 1} a_nc_n=+\infty$,
		\item $\sum_{n\geq 1} a_n^2 c_n<+\infty$.
		\end{enumerate}
	\end{lemma}
		
\begin{proof}
We assume without loss of generality that $0\leq a_n, b_n< 1$ for every $n$, and that $(a_n)_{n\in \N}$ is a non-increasing sequence.

For $j\geq 0$, let us call $D_j=\{n\geq 0: 2^{-j-1}\leq	a_n<2^{-j}\}$, and $B_j=\sum _{n\in D_j }b_n$. We call $d_j=\max(D_j)$, which is finite since $a_n\to 0$.  Observe that the integer sets $D_j$ are arranged in increasing order: $d_{j}+1  = \min( D_{j+1})$. Also, one has
$$\frac{1}{2}  \sum_{j=0}^{+\infty} 2^{-j} B_j \leq  \sum_{n\geq 0} a_nb_n = \sum_{j=0}^{+\infty} \sum_{n\in D_j } a_nb_n \leq \sum_{j=0}^{+\infty} 2^{-j} B_j,$$ 
so that $ \sum_{j=0}^{+\infty} 2^{-j} B_j=+\infty$.

We put   $n_1=0$, $j_1=1$, and $c_n=0$ for every $n\in D_0\cup D_1$.

\smallskip

Remark that  $  \sum_{n\geq d_1+1}  a_nb_n \geq 1/2 \sum_{j\geq 2} 2^{-j} B_j =+\infty$.

Let us call $n_2$ the first integer $n$ such that  $  \sum_{n=d_1+1}^{n_2}  a_nb_n >1/2$. Observing that for $n\geq d_1+1$, $a_nb_n \leq 2^{-1}$, one necessarily has $  1/2< \sum_{n=d_1+1}^{n_2}  a_nb_n <1$.

We call $j_2$ the unique integer such that $n_2\in D_{j_2}$, and we put $c_n=b_n $ for every $n \in \{d_1+1,..., n_2\}$, and $c_n =0$ for every $n \in \{n_2+1,..., d_{j_2}\}$.  By construction, 
$$  1/2< \sum_{j=j_1+1}^{j_2 } \sum_{n\in D_j} a_nc_n <1 .$$

\smallskip

We iterate the construction. Assume that we have built two finite sequences of integers $(n_k)_{k=1,...,p}$ and $(j_k)_{k=1,...,p}$  such that:
\begin{enumerate}
\item for $k=1,...,p-1$,  $j_{k+1}> j_k$, and for $k=1,...,p$, $n_k\in D_{j_k}$
\item for $k=1,...,p$, $c_n = b_n$ if $n\in \{d_{j_{k-1}}+1,...,n_k\}$, and $c_n = 0$ if $n\in \{n_k+1,...,d_{j_k}\}$,  
\item for $k=1,...,p$, one has 
\begin{equation}
\label{eq1}
1/(k+1)< \sum_{j=j_{k-1}+1}^{j_k } \sum_{n\in D_j} a_nc_n <2/k .
 \end{equation}
\end{enumerate}

Let us call $n_{p+1}$ the first integer such that  $  \sum_{n=d_p+1}^{n_{p+1}}  a_nb_n >1/(p+2)$. Observing that for $n\geq d_p+1$, $a_nb_n \leq 2^{-j_p} \leq 1/(p+1)$ (since $j_p\geq p$), one necessarily has $  1/(p+2) < \sum_{n=d_p+1}^{n_{p+1} }  a_nb_n <  1/(p+2)+1/(p+1)\leq 2/(p+1)$.

We call $j_{p+1} $ the unique integer such that $n_{p+1}\in D_{j_{p+1}}$, and we put $c_n=b_n $ for every $n \in \{d_p+1,..., n_{p+1}\}$, and $c_n =0$ for every $n \in \{n_{p+1}+1,..., d_{j_{p+1}}\}$.   Clearly, these $n_{p+1}$ and $j_{p+1}$ satisfy the recurrence properties.

Now, gathering the information, we deduce by \eqref{eq1} that
$$ \sum_{n\geq 0}a_nc_n =  \sum_{k=1}^{+\infty}  \sum_{j=j_{k-1}+1}^{j_k } \sum_{n\in D_j} a_nc_n   \geq  \sum_{k=1}^{+\infty} 1/(k+1)=+\infty$$
and, using that $a_n\leq 2^{-j} $when $n\geq D_j$, and that $j_{k-1} \geq k-1$, 
\begin{align*}
 \sum_{n\geq 0 }a_n^2 c_n & = \sum_{k=1}^{+\infty}  \sum_{j=j_{k-1}+1}^{j_k } \sum_{n\in D_j} a_n^2c_n   \leq  \sum_{k=1}^{+\infty}  \sum_{j=j_{k-1}+1}^{j_k }  2^{-j} \sum_{n\in D_j} a_nc_n \\
 &  \leq   \sum_{k=1}^{+\infty}  2^{-k+1} /(k+1)<+\infty.
 \end{align*}
 This concludes the proof.
 \end{proof}

The same lines of computations can certainly be adapted to impose   $ \sum_{n\geq 0}a_nc_n =   +\infty$ and $\sum_{n\geq 0 }  h(a_n) c_n<+\infty$ for any map $h:\R^+\to\R^+$ such that $h(x)=o(x)$ when $x\to 0^+$.

\medskip

As a first step toward Theorem \ref{thm:s-set-2}, we reduce the problem to sets that can be covered by small sets only.
 
\begin{prop}
    \label{prop-extract-1}
    Let $E \subset \R^d$ such that $\ds  \nu^s(E) = + \infty$. Then, given $\alpha >0$,  there exists a set $\Bar{E}$ such that  
     $\ds   \nu^s(\Bar{E}) = + \infty$ 
 and  $  \lim_{n\to +\infty} \sup_{t\in [0,d]} \bnt(\wE)=0$.     \end{prop}
%-----------------------PROOF----------------------%
\begin{proof} 
It is an application of Lemma \ref{lemma:convSum1}.

Call $A_n  = \sum_{k=1}^n \nu^s_k(E)$ and $\alpha_n = A_n^{-1}$. By assumption, $\alpha_n\to 0$ when $n\to +\infty$.

For every $n\geq 1$,    $S_n$  can be covered by at most   $ {2}{\alpha_n^{-1}}$ balls  of diameter $2^n \alpha_n^{1/d}$. Call $\mathcal{A}_n$ such a family of sets.  
 One obviously has 
     \begin{align*}
        \nu_n^s(E) \leq  \sum_{A \in \mathcal{A}_n} \nu_n^s(E \cap A) 
    \end{align*}
    Thus there must exist $A_n \in \mathcal{A}_n$ such that $ \nu_n^s(E \cap A_n) \geq  \alpha_n  {\nu_n^s(E)}  $.    Then one   defines the set $\widetilde {E}$ as  $$
        \widetilde{E} = \bigcup _{n\geq 1} E \cap A_n .
        $$ 
    By Lemma \ref{lemma:convSum1},  
    \begin{align*}
        \sum_{n\geq0} \nu^s_n(\widetilde{E}) \geq  \sum_{n\geq0} \nu_n^s(E \cap A_n) \geq \sum_{n\geq0} \alpha_n {\nu_n^s(E)} =+\infty.
    \end{align*}
Now, it is clear that for every $n$, $|\widetilde{E}\cap S_n | \leq 2^n\alpha_n^{1/d}$, so  by Definition \ref{eq:b_s^n}, for every $t>0$
    \begin{align*}
        \bnt(\widetilde{E}) \leq \alpha_n^{1/d}. 
    \end{align*}
 Actually, this implies  more: necessarily $\nu_n^s(\widetilde E)\leq    \alpha_n^{s/d} $.  In particular,  $\bnt (\widetilde{E}) \rightarrow 0$ as $n \rightarrow +\infty$ uniformly in $t$.

\end{proof}
%-----------------------PROOF----------------------%

Finally, we prove Theorem \ref{thm:s-set-2}.

\begin{proof}
Let $E$ be such that $\nu^s(E)=+\infty$. By Proposition \ref{prop-extract-1}, one also assumes that $\lim_{n\to +\infty} \sup_{s\in [0,d]} \bns(\wE)=0$, and that item (3) holds for some $\alpha>0$.  This two facts will not be used in this proof, but will be key in the next section.

\smallskip
 
Observe that since for every $n$ $\nu^s_n(E) \leq 1$, then $\ds A_n := \sum_{k=0}^{n} \nu_k^s (E) \leq n$.

The idea consists in  replacing $E$ by  a set $\widetilde E$ such that $\nu_n^s(\wE) \sim b_n \nu_n^s(E)$, such that $\sum_{n\geq 1} \nu_n^s(\wE) <+\infty$ but $b_n$ is "as large as possible". Lemma \ref{lemma:convSum1} helps to build such a sequence.

\smallskip

First, for every $\ep>0$, denote by 
\begin{align*}
    B_n^{\ep}= \sum_{k \geq n} \dfrac{\nu_k^s (E)}{A_k^{1+\ep}}
\end{align*}
By Lemma \ref{lemma:convSum1}, one knows that  $
   B_n^{{\ep}} \to 0  $ as $ n \to \infty$, for every $\ep>0$.

We build iteratively a non-increasing sequence $(\ep_n)_{n\geq 0} \subset \R^+$, and a sequence of integers $(n_k)_{k\geq 1}$.

 Consider $n_1$ as  the smallest  positive integer such that $B_{n_1}^{\frac{1}{4}} \leq 1$ and set $\ep_n= \frac{1}{2}$ for all $0 \leq n \leq n_1$.  
 
 Next we proceed by induction to build $(\ep_n)_{n \geq 0}$ and $(n_k)_{k\geq 1}$.
 
 Assume that   $n_1<n_2<...<n_p$ are defined.  
 
 Define $n_{p+1}$ as  the smallest integer such that 
\begin{align}
\label{ineqB}
    n_{p} < n_{p+1}  \mbox{ and } B_{n_{p+1} }^{\frac{1}{2^{p}}} \leq \dfrac{1}{2^{p}}.
\end{align}
Put  $\ep_n=\dfrac{1}{2^{p+1} }$ for all $n_{p} < n \leq n_{p+1}$. Finally, let 
\begin{align}
\label{b_n}
    b_n= \min \left\{1/2,{(A_n)^{-(1+\ep_n)}}\right\}.
\end{align}
Then by construction of $\ep_n$, one has:
\begin{itemize}
	\item[(i)] $\ep_n \rightarrow 0$ as $n \rightarrow +\infty$,
	\item[(ii)]  By \eqref{ineqB}, and the fact that $ B_{n_k}^{\frac{1}{2^{k+1}}} \leq  B_{n_k}^{\frac{1}{2^{k}}}\leq 2^{-k-1}$,
	  \begin{align}
	  \noindent
	     \sum_{n \geq 0} b_n \nu_n^s(E)  & \leq   \sum_{n \geq 0} \dfrac{\nu_n^s(E)}{A_n ^{1+ \ep_n}} \leq  \sum_{n= 0}^{n_1} \dfrac{\nu_n^s(E)}{A_n ^{1+\frac{1}{2}}} + \sum_{k \geq1} \sum_{n= n_k+1}^{n_{k+1}} \dfrac{\nu_n^s(E)}{A_n ^{1+ \frac{1}{2^{k+1}}}}  \\ 
	     & \leq 
	 \sum_{n= 0}^{n_1} \dfrac{\nu_n^s(E)}{A_n ^{\frac{3}{2}}} + \sum_{k \geq 1} B_{n_k}^{\frac{1}{2^{k+1}}}  \leq 
	 \sum_{n= 0}^{n_1} \dfrac{\nu_n^s(E)}{\left(A_n \right)^{\frac{3}{2}}} + \sum_{k \geq 1} \dfrac{1}{2^{k-1}} < + \infty. \label{eq2}
  \end{align}
\end{itemize}

\smallskip

Next, we  construct a set $\widetilde{E} \subset E$ such that for all $n\in \N$, one has 
\begin{align*}
    \lvert \nu^{s}_n(\widetilde{E})-b_n\nu^{s}_n(E)\rvert \leq 2^{-ns}.
\end{align*}
To achieve this, observe that by Definition~\ref{S_n}, $S_n $ contains a finite number of lattice points, and denote by $M_{n,d}$ their cardinality. These points are denote by $x_i$ for $i \in \{ 1, \dots, M_{n,d} \}$.

Consider the following function:  
\begin{align*}
	g_n: \left\{0,1, \dots, M_{n,d} \right\} &\longrightarrow \R^+ \\
	m &\longmapsto  \nu^s_n\left(\bigcup_{i=1}^m E \cap B(x_i,1)\right).
\end{align*}
where $g _n(0)=0$ by convention. It is clear that   $g _n$ is non-decreasing, and ranges from   $0$ to $\nu_n^s(E)$. 
Moreover, for all $m \in \left\{1, \dots, M_{n,d}-1 \right\}$, if $\left\{B(y_j,r_j) \right\}_{j=1}^{p}$ is an $s$-optimal cover  of $\ds \bigcup_{i=1}^m E \cap B(x_i,1)$, then $\left\{\left(B(y_j,r_j)\right)_{j=1}^{p}, B(x_{m+1},1) \right\}$ is a proper cover of $\ds \bigcup_{i=1}^{m+1} E \cap B(x_i,1)$  (not necessarily optimal). Using these two covers, one gets
\begin{align*}
	g_{n}(m+1)-g_{n}(m) \leq \left(\sum_{j=1}^{p} \left(\dfrac{r_j}{2^n}\right)^s + \dfrac{1}{2^{ns}}\right)-\sum_{j=1}^{p} \left(\dfrac{r_j}{2^n}\right)^s \leq 2^{-ns}.
\end{align*}
Hence,  $g_{n}$ has only small increments.  

Recalling \eqref{b_n},   $0 = g_n(0)\leq b_n  \nu^{s}_n(E) \leq  \nu^{s}_n(E)= g_n(M_{n,d})$, so there must exist  an integer $m_n \in \{1, \dots, M_{n,d} \} $ such that 
\begin{align*}
	b_n\nu^{s}_n(E) \leq  g_n(m_n) \leq b_n\nu^{s}_n(E) + 2^{-ns}.
\end{align*}
Put 
\begin{align}
    \ds \widetilde{E}_n = \bigcup_{i=1}^{m_n} E \cap B(x_i,1) \mbox{ and } \ds \widetilde{E}= \bigcup_{n \geq 0} \widetilde{E}_n.
\end{align}
Then by construction,  $\widetilde{E} \subset E$, and for all $n \in \N$one has
\begin{align*}
	b_n\nu^{s}_n(E) \leq    \nu^{s}_n(\widetilde{E})  \leq b_n\nu^{s}_n(E) + 2^{-ns}.   \end{align*}
And so, by \eqref{eq2},
\begin{align*}
 \ds  \nu^s(\wE) = \sum_{n \geq 0} \nu^{s}_n(\widetilde{E}) \leq \sum_{n \geq 0} \left(b_n \nu^{s}_n(E) +  2^{-ns}\right) < +\infty.   
\end{align*}
To complete the proof, it is enough to show that for all $\ep > 0 $, $\ds  \nu^{s-\ep}(\widetilde{E})= +\infty$. To this end, fix $\ep>0$, and let $\left(B(x_i,r_i)\right)_{i=1}^{m}$ be an optimal $(s-\ep)$-cover of $\widetilde{E} \cap S_n$, and assume    that for this specific cover, $\beta ^{s-\ep}_n(\widetilde{E})$ is reached, i.e. there exists $i\in \{1,...,m\}$ such that $r_i = 2^n \beta ^{s-\ep}_n(\widetilde{E})$. In particular, $\nu_{n}^{s-\ep}(\widetilde{E}) \geq (\beta ^{s-\ep}_n(\widetilde{E}))^{s-\ep}$.

One sees that 
\begin{align}
\label{eq:lowerbd}
	\nu_{n}^{s-\ep}(\widetilde{E})= \sum_{i=1}^{m } \left(\dfrac{r_i}{2^n}\right)^{s-\ep} \geq \sum_{i=1}^{m} \left(\dfrac{r_i}{2^n}\right)^{s} \cdot (\beta^{s-\ep}_{n}(\widetilde{E}))^{-\ep} \geq (\beta^{s-\ep}_{n}(\widetilde{E}))^{-\ep} \cdot \nu_n^s(\widetilde{E}) .
\end{align}
Two cases are separated. 

On the one hand, If $\beta^{s-\ep}_n(\widetilde{E}) \leq \sqrt[s]{\dfrac{\nu_n^s(E)}{A_n }}$, then      \eqref{eq:lowerbd}  yields 
\begin{align}
	\label{eq:lower_bound1}
	\nu_{n}^{s-\ep}(\widetilde{E})  & \geq 
     \left(\dfrac{A_n }{\nu_n^s(E)}\right)^{\ep/s}\cdot \nu_n^s(\widetilde{E}) \geq
 \left(\dfrac{A_n}{\nu_n^s(E)}\right)^{\ep/s} \cdot b_n \cdot \nu_n^s(E) \\  &\geq \nonumber
	 \dfrac{\left(\nu_n^s(E)\right)^{1-\ep/s}}{A_n^{1+ \ep_n-\ep/s}}\geq 
	 \dfrac{\nu_n^s(E)}{A_n^{1+ \ep_n-\ep/s}}.
\end{align}
where the fact that $\nu_n^s(E) \leq 1$ has been used in the last step. 

On the other hand, if $\beta^{s-\ep}_n(\widetilde{E}) \geq \sqrt[s]{\dfrac{\nu_n^s(E)}{A_n}}$, one has
\begin{align}
	\label{eq:lower_bound2}
	\nu_{n}^{s-\ep}(\widetilde{E}) \geq (\beta ^{s-\ep}_n(\widetilde{E}))^{s-\ep} \geq \dfrac{\left(\nu^s_n(E)\right)^{1 - \ep/s}}{A_n ^{1 - \ep/s}} \geq \dfrac{\nu^s_n(E)}{A_n^{1 - \ep/s}}
\end{align} 
Finally,  using the fact that $\ep_n \to 0$ together with the lower bounds \eqref{eq:lower_bound1} and \eqref{eq:lower_bound2}, one gets that for every large $n$,  $ \nu_{n}^{s-\ep}(\widetilde{E})  \geq \dfrac{\nu^s_n(E)}{A_n}$.     By Lemma \ref{lemma:convSum1},  $\ds \sum_{n \geq 0} \dfrac{\nu^s_n(E)}{A_n} = +\infty$, hence  $\nu^{s-\ep}(\wE) =\sum_{n\geq 0} \nu_{n}^{s-\ep}(\widetilde{E})  =+\infty$.

This holds for every   $\ep>0$,   so $\dhaus{\widetilde{E}}=s$.
\end{proof}

%%%%%%%%%%%%%%%%%%%%%%%%%%%%%%%%%%%%%%%%%%%%%%%%%%%%%%%
%%%%%%%%%%%%%%%%%% Potential Methods %%%%%%%%%%%%%%%%%%
%%%%%%%%%%%%%%%%%%%%%%%%%%%%%%%%%%%%%%%%%%%%%%%%%%%%%%%
\section{Potential Methods}
\label{sec:potentialMethods}

\subsection{First part of Theorem \ref{thm:potential}}

Consider $E\subset \R^d$, and assume that  there exists a Radon measure $\mu $ on $\R^d$  such that $\mu(E)=+\infty$ and 
		$\ds \sum_{n \geq 0} 2^{ns} I_s(\mu_{|S_n}) < +\infty.$ We prove that    $\nu^s(E)=+\infty$, which implies that $\wnu^s(E)=+\infty$ and     $\dhaus{E}\geq s$.  
		 
 For  $n \in \N$, we write $\mu_n=\mu_{|S_n}$, and define
		\begin{align*}
			& \phi^s_{\mu_n}:= \int_{\R^d} \dfrac{d\mu_n(y)}{\norm{x-y}^s \vee 1 } \quad \mbox{ and }  \ds E_n = \left\{x \in E\cap S_n: \, \max_{r \geq 1} \dfrac{\mu_n\left(B(x,r)\right)}{\left(\frac{r}{2^n}\right)^s} \leq 1\right\}
		\end{align*}
For every $x \in E_n^c$,   there exists an integer $r_x$ such that $\dfrac{\mu_n\left(B(x,r_x)\right)}{\left(\frac{r_x}{2^n}\right)^s} \geq 1$. One has 
		\begin{align*}
			\phi^s_{\mu_n}(x)= \int_{\R^d} \dfrac{d\mu_n(y)}{\norm{x-y}^s \vee 1 } \geq \int_{B(x,r_x)} \dfrac{d\mu_n(y)}{\norm{x-y}^s \vee 1 }\geq \dfrac{\mu_n\left(B(x,r_x)\right)}{r_x^s} \geq \frac{1}{2^{ns}}.
		\end{align*}
		Then $\ds I_s(\mu_n) \geq \int_{E_n^c} \phi^s_{\mu_n}(x) d\mu_n(x) \geq \frac{1}{2^{ns}} \mu_n(E_n^c)$, which implies that 
		\begin{align*}
			\ds \sum_{n \geq 0} \mu_n(E_n^c) \leq \sum_{n\geq 0} 2^{ns} I_s(\mu_n) < +\infty.
		\end{align*} 
		But as $E \cap S_n = E_n \cup E_n^c$ and $\ds \sum_{n\geq 0} \mu_n(E \cap S_n) = +\infty$, then $\ds \sum_{n\geq 0} \mu_n(E_n) = +\infty$.
		Moreover, by Proposition \ref{prop_mu_n} a), one has $\nu^s_n(E_n) \geq  \frac{\mu_n(E_n)}{2^s}$. 
		Finally,  $\nu^s(E) =  \sum_{n \geq 0} \nu^s_n(E_n) = + \infty$ which gives that $\dhaus{E}\geq s$.
		%-----------------------------------------------%

\subsection{Second part of Theorem \ref{thm:potential}} 

This is the most delicate part. Assume now that  $  \wnu^s(E) = + \infty$, and fix $0<\ep<s$.

Our goal is to build  a Radon  measure $\mu^{\ep}$ on $\R^d$ such that  $ \mu^{\ep}(E )= +\infty$ and  $\ds \sum_{n\geq 0} 2^{n(s - \ep)} I_{s -\ep}(\mu_{|S_n}^{\ep}) < +\infty$.
 We are going to build each measure $\mu_n^\ep =\mu_{|S_n}^{\ep}$.

 For this, we use the results we previously proved.
 
By Theorem \ref{thm:s-set-2} there exists a set $E_1 \subset E$ such that  $\lim_{n\to +\infty} \sup_{t\in [0,d]} \bnt(E_1)=0$ and $  \nu^s(E_1) =+\infty$.

	Then by Theorem \ref{thm:s-set}, there exists a macroscopic  $s$-set $E_2 \subset E_1$ such that $\dhaus{E_2}= s$ and 
  $\ds  \nu^s(E_2) <+\infty$.

Consider  an optimal   $(s-\frac{\ep}{2})$-cover $\{B(x_i,r_i)\}_{i=1}^m$   of $E_2 \cap S_n$. One sees that 
    \begin{align*}
     \left(\beta^{s-\ep/2}_n(E_2)\right)^{\frac{\ep}{4}} \nu^{ s-\frac{\ep}{2} }_n(E_2 )  =
    	&  \left(\beta^{s-\ep/2}_n(E_2)\right)^{\frac{\ep}{4}} \sum_{i=1}^m \left( \dfrac{r_i}{2^n}\right)^{s-\frac{\ep}{2}} \\ = & \left(\beta^{s-\ep/2}_n(E_2)\right)^{\frac{\ep}{4}} \sum_{i=1}^m \left( \dfrac{r_i}{2^n}\right)^{s-\frac{\ep}{4}} \left( \dfrac{r_i}{2^n}\right)^{-\frac{\ep}{4}} \\ \geq
    	& \sum_{i=1}^m \left( \dfrac{r_i}{2^n}\right)^{s-\frac{\ep}{4}} \geq 
    	 \nu^{ s-\frac{\ep}{4}}_n(E_2 ),
    \end{align*}
    where we used that $ \beta^{s-\ep/2}_n(E_2) \geq  \dfrac{r_i}{2^n}$. 
    Recalling that  $\dhaus{E_2}= s$, it follows that   $\ds \sum_{n \geq}  \left(\beta^{s-\ep/2}_n(E_2)\right)^{\frac{\ep}{4}} \nu^{ s-\frac{\ep}{2} }_n(E_2 ) = +\infty $. Moreover as $E_2 \subset E_1$, then $\beta^{s-\ep/2}_n(E_2) \rightarrow 0$ as $n \rightarrow +\infty$.

   Setting $a_n = \left(\beta^{s-\ep/2}_n(E_2)\right)^{\frac{\ep}{4}}$ and $b_n=\nu_{ s-\frac{\ep}{2} }^n(E_2 )$, one then sees that the sequences $(a_n)_{n\geq 1}$ and $(b_n)_{n\geq 1}$  satisfies the assumptions of     Lemma \ref{lemma:convSum2}. Consider the sequence $(c_n)_{n\geq 1}$ given by this Lemma, and define the set $E_3\subset E_2$ as follows: for every $n\geq 1$,
   \begin{itemize}
   \item 
   if $c_n=0$, then $E_3\cap S_n =\emptyset$,
   
   \item 
   if $c_n=b_n$, then $E_3\cap S_n =E_2\cap S_n$.
   \end{itemize}
   
It is immediate from the construction and   Lemma \ref{lemma:convSum2} that    $c_n = \nu^{s-\ep/2}_n(E_3)$ and 
    \begin{align}
 \nonumber       \sum_{n \geq}  \left(\beta^{s-\frac{\ep}{2}}_n(E_2)\right)^{\frac{\ep}{4}} \nu^{ s-\frac{\ep}{2} }_n(E_3 ) = +\infty \\
\label{eq:4}
     \mbox{ and }    \sum_{n \geq}  \left(\beta^{s-\frac{\ep}{2}}_n(E_2)\right)^{\frac{\ep}{2}} \nu^{ s-\frac{\ep}{2} }_n(E_3 ) < +\infty 
    \end{align}
   Finally, by Proposition \ref{propab}, there exists $\emptyset \neq E_4 \subset E_3 \subset E$ such that for all $n \in\N$,
	\begin{align}
	\label{eq:E_3}
		\frac{4}{5}\nu^{s-\frac{\ep}{2}}_n(E_3) \leq \nu^{s-\frac{\ep}{2}}_n(E_4) \leq \nu^{s-\frac{\ep}{2}}_n(E_3) \\ 
		\mbox{ and }   \ \ \ \ \ 
			\label{eq:E_4}\nu^{s-\frac{\ep}{2}}_n(E_4 \cap B(x,r)) \leq c_s \left(\frac{r}{2^n}\right)^{s-\frac{\ep}{2}}
		\end{align}
		for all $x \in \Z^d$ and $r \geq 1$.

 Define the measures  $\mu_n^{\ep}(A) := \left(\beta^{s-\frac{\ep}{2}}_n(E_2)\right)^{\frac{\ep}{4}} \nu^{s-\frac{\ep}{2}}_n(E_4 \cap A)$.  Then  by our construction and \eqref{eq:E_3}, one has 
    \begin{align*}
	    \sum_{n\geq0} \mu_n^{\ep}(E \cap S_n) =  & \sum_{n\geq0} \left(\beta^{s-\frac{\ep}{2}}_n(E_2)\right)^{\frac{\ep}{4}} \nu^{s-\frac{\ep}{2}}_n(E_4 ) \\\geq 
	    & \frac{4}{5} \sum_{n\geq0} \left(\beta^{s-\frac{\ep}{2}}_n(E_2)\right)^{\frac{\ep}{4}} \nu^{ s-\frac{\ep}{2} }_n(E_3 ) = +\infty.
    \end{align*} 
We are left to prove that 
\begin{align*}
	\sum_{n\geq 0} 2^{n(s-\ep)} I_{s -\ep}(\mu_n^{\ep}) = \sum_{n\geq 0} 2^{n(s-\ep)} \int_{\R^d} \phi^{s -\ep}_{\mu_n^{\ep}}(x) d\mu_n^{\ep}(x)< +\infty
\end{align*}
For $x \in S_n$, one can write \begin{align*}
	\phi_{s-\ep}^{\mu_n^{\ep}}(x)= 
	&\int_{S_n} \dfrac{d\mu_n^{\ep}(y)}{\norm{x-y}^{s-\ep} \vee 1}
\end{align*}
Every   $y \in S_n$ belongs to the ball $ B(x,2^{n+1})$. For $1 \leq r \leq 2^{n+1}$, denote by $ \ds m_n^{\ep}(r)= \mu_n^{\ep}(B(x,r))$. By \eqref{eq:E_4}, one has 
\begin{align}
	\label{eq:m_n}
	m_n^{\ep}(r) = 
	& \left(\beta^{s-\frac{\ep}{2}}_n(E_2)\right)^{\frac{\ep}{4}}\nu^{s-\frac{\ep}{2}}_n\left( E_4 \cap B(x,r)\right) \leq c_s \,\left(\beta^{s-\frac{\ep}{2}}_n(E_2)\right)^{\frac{\ep}{4}} \left(\frac{r}{2^{n}}\right)^{s-\frac{\ep}{2}}.
\end{align}
Using the fact that  $\ds B(x,2^{n+1}) = \bigcup_{r=1}^{2^{n+1}} B(x,r) \setminus B(x,r-1)$, one has 
\begin{align*}
	\phi_{s-\ep}^{\mu_n^{\ep}}(x)  & \leq
 \sum_{r=1}^{2^{n+1}} \int_{B(x,r) \setminus B(x,r-1)} \dfrac{d\mu_n^{\ep}(y)}{\norm{x-y}^{s-\ep} \vee 1} \\
 & = \mu_n^\ep (B(x,1))+ 
	 \sum_{r=2}^{2^{n+1}} \int_{B(x,r) \setminus B(x,r-1)} \dfrac{d\mu_n^{\ep}(y)}{\norm{x-y}^{s-\ep}} .
\end{align*}
One the one hand, by \eqref{eq:E_3}, $\mu_n^\ep (B(x,1)) \leq  c_s \,\left(\beta^{s-\frac{\ep}{2}}_n(E_2)\right)^{\frac{\ep}{4}}  2^{-n(s-\frac{\ep}{2})}$. On the other hand, 
\begin{align*}
& \sum_{r=2}^{2^{n+1}} \int_{B(x,r) \setminus B(x,r-1)} \dfrac{d\mu_n^{\ep}(y)}{\norm{x-y}^{s-\ep}} 	\\
&= 	  \sum_{r=2}^{2^{n+1}} \int_{r-1}^r t^{\ep -s} d m_n^\ep(t) \\
& = 
	  \sum_{r=2}^{2^{n+1}} \left( \left[ t^{\ep -s} m_n^\ep(t)\right]^{r}_{r-1} + (s-\ep)\int_{r-1}^r t^{\ep -s -1} m_n^\ep(t) dt \right) \\ & \leq 
	   c_s \,\left(\beta^{s-\frac{\ep}{2}}_n(E_2)\right)^{\frac{\ep}{4}} \, 2^{-n(s-\frac{\ep}{2})} \sum_{r=1}^{2^{n}}  \left( \left[ t^{\frac{\ep}{2}} \right]^{r}_{r-1} + (s-\ep)\int_{r-1}^r t^{\frac{\ep}{2}-1}dt \right) \\
&  \leq 
	  c_s \left( 1+ 2\frac{s-\ep}{\ep}\right) \left(\beta^{s-\frac{\ep}{2}}_n(E_2)\right)^{\frac{\ep}{4}} \, 2^{-n(s-\frac{\ep}{2})} \sum_{r=1}^{2^{n+1}} \left( r^{\frac{\ep}{2}} - (r-1)^{\frac{\ep}{2}}\right) \\
&  \leq
	  C  \left(\beta^{s-\frac{\ep}{2}}_n(E_2)\right)^{\frac{\ep}{4}} \, 2^{-n(s-\ep)}.
\end{align*}
for some constant $C$.  So 
$$\phi_{s-\ep}^{\mu_n^{\ep}}(x) \leq c_s ,\left(\beta^{s-\frac{\ep}{2}}_n(E_2)\right)^{\frac{\ep}{4}}  2^{-n(s-\frac{\ep}{2})}+ C  \left(\beta^{s-\frac{\ep}{2}}_n(E_2)\right)^{\frac{\ep}{4}} \, 2^{-n(s-\ep)}\leq \widetilde C   \left(\beta^{s-\frac{\ep}{2}}_n(E_2)\right)^{\frac{\ep}{4}} \, 2^{-n(s-\ep)}.$$
 Moving to the integral, one gets  
\begin{align*}
	\ds I_{s -\ep}(\mu_n^{\ep}) = \int_{\R^d} \phi_{s-\ep}^{\mu_n^{\ep}}(x) d\mu_n^{\ep}(x) \leq  C \left(\beta^{s-\frac{\ep}{2}}_n(E_2)\right)^{\frac{\ep}{4}} 2^{-n(s-\ep)} \mu_n^{\ep}(E_4).
	\end{align*}
Finally,  recalling \eqref{eq:4}, \eqref{eq:E_3}, \eqref{eq:E_4} and the definition of $\mu_n^\ep$, one has
\begin{align*}
	\sum_{n\geq 0} 2^{n(s-\ep)} I_{s-\ep}(\mu_n^{\ep}) &  \leq   C\sum_{n\geq 0}  \left(\beta^{s-\frac{\ep}{2}}_n(E_2)\right)^{\frac{\ep}{4}} \mu_n^{\ep}(E_4) \\
	& \leq  \, C \sum_{n\geq 0} \left(\beta^{s-\frac{\ep}{2}}_n(E_2)\right)^{\frac{\ep}{2}} \nu_{ s-\frac{\ep}{2}}^n(E_4) < +\infty
\end{align*}
as desired. 
%----------------------PROOF----------------------%

%%%%%%%%%%%%%%%%%%%%%%%%%%%%%%%%%%%%%%%%%%%%%%%%%%%%%%%
%%%%%%%%%%%%%%%% Projection of a Set %%%%%%%%%%%%%%%%%%
%%%%%%%%%%%%%%%%%%%%%%%%%%%%%%%%%%%%%%%%%%%%%%%%%%%%%%%

\section{Projection of a Set}
\label{sec:projection}
In this section we are considering the orthogonal projection of sets in $\R^2$ and we aim at proving  the projection Theorem \ref{thm:projection} for the macroscopic Hausdorff dimension. 

Let  us introduce some notations.

For every $\theta\in [0,2\pi]$, call  $e_\theta=(\cos \theta,\sin\theta)$ the vector with angle $\theta$, and $L_\theta$  the straight  line in $\R^2$ with angle $\theta$    passing through the origin.

Then, recall that  $proj_\theta:\R^2 \to L_\theta$ is  the orthogonal projection onto $L_\theta$.

\subsection{Case where $\dhaus{E}\geq  1$}

Let us start by proving item b) of Theorem \ref{thm:projection}, assuming that item a) is proved.
 
 Consider $E\subset \R^2$ with $\dhaus{E} \geq 1$.
 
 By Theorem \ref{thm:s-set-2},  for every $p\geq 2$, there exists $E_p\subset E$ such that $\dhaus{E_p}=1-1/p$.
 For each set $E_p$, by item a), there exists a set $\Theta_p\subset  [0,\pi]$ of full Lebesgue measure such that for every $\theta \in \Theta_p$, $\dhaus{proj_\theta(E_p)}=1-1/p$. In particular, this implies that $\dhaus{proj_\theta(E)}\geq 1-1/p$. 
 
 Consider now the set $\Theta=\bigcap_{p\geq 2}\Theta_p$. The above arguments show that $\Theta$ is still of full Lebesgue measure in $[0,\pi]$, and that for every $\theta\in \Theta$, $\dhaus{proj_\theta(E)}\geq 1$.  Since obviously $\dhaus{proj_\theta(E)}$ is always less than 1 (since it is included in $L_\theta$), the result follows.
 
%----------------------THEOREM----------------------%

\subsection{First extractions when   $\dhaus{E}< 1 $}
 
Fix a set $E\subset \R^2 $ with  $0<\dhaus{E} =s< 1$.  The rest of the section is devoted to prove that   $\dhaus{\proj E} = \dhaus{E}$ for almost every $\theta \in [0,\pi]$.

 \medskip

Writing $L_\theta = \{ \lambda e_\theta: \lambda\in \R\}$, we can define the $n$-th shells inside $L_\theta$ as $S_n^\theta =\{v= (x,y)\in L_\theta: \|v\|_2\in [2^{n-1},2^n]\}$.  Identifying $L_\theta$ with $\R$, the results we obtained before in dimension 1 apply to $L_\theta$ and $S_n^\theta$.

We are going to project 2-dimensional measures onto the lines $L_\theta$. For this, let us define for every $n\geq 0$  the cylinders
\begin{align}
\label{defCn}
    C_n^\theta := \proj^{-1} S_n^\theta.
\end{align}
%----------------------THEOREM----------------------%
 
\medskip

We are going to prove that for every $0<\ep<s$, the set 
\begin{equation}
\label{thetaep}
\Theta_{s-\ep}=\{\theta\in [0,\pi]: \dhaus{proj_\theta(E)} \geq s-\ep\}
\end{equation}
  has full Lebesgue measure. The conclusion then follows  using the same argument as the one used to prove item b). More precisely, from the properties above, $\Theta :=\bigcap_{p\geq 1} \Theta_{s-1/p}$  has full Lebesgue measure, and for every $\theta \in \Theta$, $\dhaus{proj_\theta(E)}\geq s$. But  since $proj_\theta$ is a Lipschitz mapping,   $\dhaus{proj _\theta (E)} \leq s=\dhaus{E}$. Finally one gets $\dhaus{\proj E} = \dhaus{E}$ for almost all $\theta \in [0,\pi]$.

\medskip

Fix  $0<\ep<s$. 

Applying Theorem \ref{thm:potential}(2),  there exists a  Borel measure $\mu^\ep$ supported by $E   $ such that  	
\begin{align}
\label{eq11} \sum_{n\geq0} \mu_{n}^{\ep}(E \cap S_n)= +\infty,\\
\label{eq12} \mbox{ and } \ \ \    \sum_{n\geq 0} 2^{n(s - \ep)} I_{s -\ep}(\mu_{n}^{\ep}) < +\infty,
\end{align}
where  $\mu_n ^\ep $ is a simplified notation for $ \mu_{|S_n}^\ep$. Observe that in fact, via the finer Theorem \ref{thm:s-set-2} and Proposition \ref{prop-extract-1},    we can impose that $\lim_{n\to +\infty} \mu_{n}^{\ep}(E \cap S_n) =0$.

    \medskip
    
We need to impose an additional condition on $\mu^\ep$, namely that     
    \begin{equation}
    \label{eq-finale2}
  \sum_{n\geq 0}  2^{-n  } \mu_n^\ep ( {E} \cap S_n)  \left(\sum_{k=0}^n 2^{k } \mu_k^\ep ( {E} \cap S_k) \right)  <+\infty.
  \end{equation}

This is achieved thanks to the following lemma.

	\begin{lemma}
		\label{lemma:convSum3}
		Let $(a_n)_{n \geq 1}$  and $(b_n)_{n \geq 1}$ be two positive sequences converging to zero, such that $\sum_{n\geq 1} a_n =+\infty$ and $\sum_{n\geq 1} a_nb_n=+\infty$. There exists a sequence $(c_n)_{n\geq 1}$ such that:
		\begin{enumerate}
		\item either $c_n =a_n$, or $c_n=0$,
		\item $\sum_{n\geq 1} c_n=+\infty$,
		\item $\sum_{n\geq 1} c_n b_n<+\infty$.
		\end{enumerate}
	\end{lemma}
		
\begin{proof}
Again, without loss of generality, we assume that $0<a_n,b_n< 1$.
Let us call $D_j=\{n\geq 0: 2^{-j-1}\leq b_n<2^{-j}\}$, for $j\geq 0$.

Put $c_n=0$ for every $n\in D_0 \cup D_1$, and $n_0=0$, $j_0=1$.

We know that $\sum_{j\geq 2} \sum_{n\in D_j} a_n b_n=+\infty$. We go through each $D_j$ in increasing order. Consider the first couple $(n_1,j_1)$ such that  $n_1\in D_{j_1}$ and $\sum_{j= 2}^{j_1-1}  \sum_{n\in D_j} a_n b_n +  \sum_{n\in D_{j_1}, n\leq n_1} a_n b_n  \geq 1/2$. Put $c_n=a_n$ for all $n\in \bigcup_{j= 2}^{j_1-1}D_j \cup \{n\in D_{j_1}: n\leq n_1\} $, and $c_n=0$ for all $n\in   \{n\in D_{j_1}: n> n_1\} $.
By our choice, 
$$1/2\leq \sum_{j= 0}^{j_1}  \sum_{n\in D_j} c_n b_n =  \sum_{j= 2}^{j_1-1}  \sum_{n\in D_j} a_n b_n +  \sum_{n\in D_{j_1}, n\leq n_1} a_n b_n  <1.$$

We then iterate the process:  assume that we have built two finite sequences of integers $(n_k)_{k=1,...,p}$ and $(j_k)_{k=1,...,p}$  such that 
\begin{enumerate} 
\item for $k=1,...,p-1$,  $j_{k+1}> j_k$, and for $k=1,...,p$, $n_k\in D_{j_k}$
\item for $k=1,...,p$, $c_n = a_n$ if $n\in \bigcup_{j= j_{k-1}}^{j_k-1}D_j \cup \{n\in D_{j_k}: n\leq n_k\} $, and $c_n=0$ for all $n\in   \{n\in D_{j_k}: n> n_k\} $. 
\item for $k=1,...,p$, one has 
\begin{equation}
\label{eq1bis}
2^{-k} \leq \sum_{j= j_{k-1}}^{j_k}  \sum_{n\in D_j} c_n b_n  < 2^{-k+1}. \end{equation}
\end{enumerate}

We know that $\sum_{j\geq j_p+1} \sum_{n\in D_j} a_n b_n=+\infty$.  Consider the first couple $(n_{p+1} ,j_{p+1})$ such that  $n_{p+1} \in D_{j_{p+1}}$ and $\sum_{j= j_{p}}^{j_{p+1}-1}  \sum_{n\in D_j} a_n b_n +  \sum_{n\in D_{j_{p+1}}, n\leq n_{p+1}} a_n b_n  \geq 2^{-(p+1)}$. Put $c_n=a_n$ for all $n\in \bigcup_{j= j_p}^{j_{p+1}-1}D_j \cup \{n\in D_{j_{p+1}}: n\leq n_{p+1}\} $, and $c_n=0$ for all $n\in   \{n\in D_{j_{p+1}}: n> n_{p+1}\} $. Then, since for all the selected integers $n$, $a_nb_n\leq 2^{-j_{p+1}}\leq 2^{-(p+1)}$,  \eqref{eq1bis} holds true.

\medskip

Collecting the information, on one hand one has by \eqref{eq1bis} 
$$\sum_{n\geq 0} c_n b_n = \sum_{k\geq 1}  \sum_{j= j_{k-1}}^{j_k}  \sum_{n\in D_j} c_n b_n \leq \sum_{k\geq 1} 2^{-k+1}<+\infty.$$
On the other hand, since $j_k \geq k+1$, one sees that  
for each $n\in D_j$ for $j\in \{j_{k-1},...j_k\}$, $b_n \leq 2^{-k}$,  so again by \eqref{eq1bis},
$$\sum_{n\geq 0} c_n = \sum_{k\geq 1}  \sum_{j= j_{k-1}}^{j_k}  \sum_{n\in D_j} c_n  \geq  \sum_{k\geq 1} 2^k  \sum_{j= j_{k-1}}^{j_k}  \sum_{n\in D_j} c_n b_n \geq \sum_{k\geq 1} 1=+\infty,$$
hence the result.
 \end{proof}

 Setting $a_n=\mu_n^\ep(  E)$, then   $(a_n)_{n\geq 0}$ tends to zero when $n$ tends to infinity. Define then
  $$b_n = 2^{-n  }  \sum_{k=0}^n 2^{k } a_k.$$
  Since $\sum_{k=0}^n 2^{k } \sim 2^{n }$, $(b_n)_{n\geq 0}$ is a generalized  Caesaro mean associated with the sequence $(a_n)_{n\geq 0}$, and converges to zero when $n$ tends to infinity. 
  
  So either $\sum _{n\geq 1} a_nb_n<+\infty$, and \eqref{eq-finale2} is true, or $\sum _{n\geq 1} a_nb_n =+\infty$ and we are exactly in the situation of Lemma \ref{lemma:convSum3}: there exists a sequence    $(c_n)_{n\geq 1}$ such that:
		\begin{enumerate}
		\item either $c_n =a_n$, or $c_n=0$,
		\item $\sum_{n\geq 1} c_n=+\infty$,
		\item $\sum_{n\geq 1} c_n b_n<+\infty$.
		\end{enumerate}
Setting $\widetilde {E}= \bigcup_{n\geq 0: a_n=c_n}  {E}\cap S_n $, by construction one has $\mu^\ep( \widetilde {E}) = \sum_{n\geq 1} c_n=+\infty$, and since $\mu^\ep_k(\widetilde {E}\cap S_k) =c_k\leq a_k=\mu^\ep_k(  {E} \cap S_k)  $, one has
    $$ \sum_{n\geq 0}  2^{-n  } \mu_n^\ep (\widetilde{E}  \cap S_n)  \left(\sum_{k=0}^n 2^{k } \mu_k^\ep (\widetilde{E}  \cap S_k) \right)  \leq \sum_{n\geq 1} c_n b_n <+\infty,$$
    hence \eqref{eq-finale2} is obtained for $\widetilde{E}$.
This property will be used at the very end of  the proof of Proposition \ref{prop-proj-pot} only. It is obvious that if  Theorem \ref{thm:projection} is proved for this smaller set $E$, it is also true for the original set.

\medskip
    
Finally, observe   that, replacing $\widetilde{E}$ by $\bigcup_{n\geq 0} \widetilde{E} \cap S_{2n}$ or $\bigcup_{n\geq 0} \widetilde{E} \cap S_{2n+1}$, one can assume in addition to \eqref{eq11}, \eqref{eq12}  and \eqref{eq-finale2} that 
\begin{equation}
\label{eq-separation}
\mbox{if } S_n\neq \emptyset, \ \mbox{ then } S_{n-1}=S_{n+1}=\emptyset.
\end{equation}

To resume this section, we have proved that the original set $E$ contains a subset, still denoted by $E$ for simplification,  and a measure $\mu^\ep$ supported by $E$ such that \eqref{eq11}, \eqref{eq12}, \eqref{eq-finale2} and \eqref{eq-separation} simultaneously hold. 

\subsection{Final proof of item a) of Theorem \ref{thm:projection}}

Consider the set $E$ obtained after extraction above.
 For all $\theta\in [0,\pi]$, $k \geq n$ and $A \subset L_\theta$, we focus on  the restriction of $\mu_k^\ep$ on $C_n(\theta)$ 
 \begin{align*}
 (	\mu_k^{\ep} )_{|C_n^\theta}(A) := \mu_k^{\ep} (\left\{x \in E \cap S_k \, :\, \proj x\in A \cap S_n^\theta \right\}),
 \end{align*}
Equivalently for each non-negative function $f$, one has 
 \begin{align*}
 	\int_{-\infty}^{+\infty} f(t) d (\mu_k^{\ep}  )_{|C_n^\theta}(t) \, =\, \int_{C_n^\theta \cap S_k} f  (x.e_\theta) d\mu_k^{\ep}(x).
 \end{align*}
 where   $x.e_\theta$ denotes the scalar product. Since $e_\theta$ is unitary,   we identify $x.e_\theta $ with $ \proj x $,   the orthogonal projection  of $x$ onto $ L_\theta$. 
 
 \begin{defn}
 The projected measure $\mu^{\ep,\theta}$ is defined as $ \mu^{\ep,\theta} =\sum_{n\geq 1} \mu_n^{\ep,\theta}$, where 
\begin{align}
\label{defmuep}
	\mu_n^{\ep,\theta} = \sum_{k \geq n}   (\mu_k^\ep)_{|C_n^\theta}.
\end{align}
\end{defn}
Note that  each $\mu_n^{\ep,\theta}$ is a measure supported on $\proj E \cap S_n^\theta$.

We are going to prove that  for almost all $\theta \in [0,\pi]$,
 \begin{align}
 \label{eq:dim_hproj}
 	\sum_{n\geq 0} \mu_n^{\ep,\theta}(\proj E  ) = +\infty \mbox{ and  } \sum_{n\geq 0} 2^{n (s-\ep)} I_{s-\ep}(\mu_n^{\ep,\theta}) < \infty.
 \end{align}
 for almost all $\theta \in [0,\pi]$. Then item a) of Theorem \ref{thm:potential} will allow us to conclude that the set $\Theta_{s-\ep}$ defined by \eqref{thetaep} has full Lebesgue measure, as announced.

 This is the purpose of the next two propositions.
 \begin{prop}
  For every $\theta\in[0,\pi]$, 
 \begin{equation}
 \label{ineg20}
  \mu^{\ep,\theta}(\proj E  ) = +\infty.
 \end{equation}
 \end{prop} 
\begin{proof}
This simply follows from the observation that 
$$   \mu^{\ep,\theta}(\proj E  ) =	\sum_{n\geq 0} \mu_n^{\ep,\theta}(\proj E  )= \sum_{n\geq 0}   \sum_{k \geq n}   (\mu_k^\ep)_{|C_n^\theta} (E) \geq  	\sum_{n\geq 0} \mu_n^{\ep}(  E  )=+\infty,$$
since the union of the $(C_n^\theta)_{n\geq 1}$ cover $\R^2$ (there are small overlaps (their borders) between the $C^\theta_n$). Hence the result.
\end{proof}

So the first part of \eqref{eq:dim_hproj} is proved.

Let us move to the second part. Observe that even if $ \mu^{\ep,\theta}(\proj E  ) = +\infty$, it is likely that $\proj E$ has dimension less than $\dhaus{E}$. A trivial example is when the $s$-dimensional set  $E$ is included in a straight line of angle $\phi$ passing through 0, and $\theta=\phi+\pi/2$.

 \begin{prop}
 \label{prop-proj-pot} One has
 \begin{equation}
 \label{ineg21}
\E_\theta \left[\sum_{n\geq 0} 2^{n (s-\ep)} I_{s-\ep}(\mu_n^{\ep,\theta}) \right]  <+\infty . \end{equation}
 \end{prop}
 
\begin{proof} Remark that if \eqref{ineg21} is proved, then $\sum_{n\geq 0} 2^{n (s-\ep)} I_{s-\ep}(\mu_n^{\ep,\theta}) <+\infty$ for Lebesgue almost every $\theta\in [0,\pi]$, so \eqref{eq:dim_hproj} and item a) of Theorem \ref{thm:projection} are proved.

\medskip

 We start with the following lemma.
	%----------------------LEMMA----------------------%
	\begin{lemma}
		\label{lemma:n,k} There exists a constant $C_0>0$ such that  the following holds.
		Let $x \in S_k$ for some $k \geq 0$. For all $0\leq n \leq k$, the set   $J_{x,n}=\{\theta\in [0,\pi] : x \in C_k^\theta\}$ is an interval modulo $\pi$, and   $|J_{x,n}| \leq  C_0 2^{n-k} $.
	\end{lemma}
	%----------------------LEMMA----------------------%
	%----------------------PROOF----------------------%
	\begin{proof}
	
	The fact that $J_{x,k}$ is an interval is obvious.
	
		Let $x=(u,v) \in S_k$.  We   study the case where $x_1\geq 0$, the case $x_1 < 0$ being symmetric.  Using polar coordinates, one has $x=(r \cos \theta_0, r \sin \theta_0)$ for some $2^{k-1} \leq r \leq 2^{k}$ and $\theta_0 \in [-\frac{\pi}{2},\frac{\pi}{2}]$. Then the projection of $x$ on $L_\theta$ is given by: 
		\begin{align*}
			\proj x = ( r \cos (\theta - \theta_0) \cos \theta , r \cos (\theta - \theta_0) \sin \theta ).
		\end{align*}
		Recall  \eqref{defCn}, one sees that  for $0\leq n\leq k$, 
		\begin{align*}
			x \in C_n^\theta  & \iff 
			  2^{n-1} \leq r \cos (\theta - \theta_0) \leq 2^{n} \\
			  &  \iff 
			  \dfrac{2^{n-1}}{r} \leq  \cos (\theta - \theta_0) \leq \dfrac{2^{n}}{r} \\ 
			  & \iff   2^{n-k} \leq \cos (\theta - \theta_0) \leq \min \{1, 2^{n-k+1}\} \\
			  & \iff 
			\theta \in \left[ \theta_0 +\arccos\left(  2^{n-k}\right), \theta_0 + \arccos\left(\min \{1, 2^{n-k+1}\} \right)\right] \mod \pi .
		\end{align*}
		Denote by $J_{n,x}:= \left[  \theta_0 +\arccos\left(\dfrac{1}{2} 2^{n-k}\right),  \theta_0 + \arccos\left(\min \{1, 2^{n-k+1}\} \right)\right]$. The Taylor development   $\arccos(y) = \frac{\pi}{2} -y + o(y)$ yields that  $|J_{n,x}| = 2^{n-k}(1+o(1))$.
	\end{proof}
	%----------------------PROOF----------------------%
	
From the proof, it also follows that 	$|J_{x,n}| \sim   C 2^{n-k} $ when $n/k$ is quite small.

\medskip

Let us study \eqref{ineg21}. One has   \begin{align*}
&  \E_\theta \left[\sum_{n\geq 0} 2^{n (s-\ep)} I_{s-\ep}(\mu_n^{\ep,\theta}) \right] \\
 	& =\int_0^\pi \left[\sum_{n\geq 0} 2^{n(s-\ep)} I_{s-\ep}(\mu_n^{\ep,\theta}) \right] d\theta \\  
 	& = \int_0^\pi \left[ \sum_{n\geq 0} 2^{n(s-\ep)}\int_{S_n^\theta} \int_{S_n^\theta} \dfrac{d \mu_n^{\ep,\theta}(u)\, d \mu_n^{\ep,\theta}(v) }{|u-v|^{s-\ep} \vee 1}\right] d\theta \\
	& = \int_0^\pi \left[ \sum_{n\geq 0} 2^{n(s-\ep)}  \sum_{j,k \geq n}    \int_{E \cap S_j \cap C_n^\theta} \int_{E \cap S_k\cap C_n^\theta} \dfrac{d \mu_k^\ep(x)\, d \mu_j^\ep(y) }{|x \cdot e_ \theta- y \cdot e_\theta |^{s-\ep} \vee 1}\right] d\theta \\
	 & := I_1 + 2I_2
 \end{align*}
 where
  \begin{align*}
 	I_1 & = \int_0^\pi \left[ \sum_{n\geq 0}  2^{n(s-\ep)}  \sum_{k\geq n} \iint_{(E \cap S_j \cap C_n^\theta)^2}  \dfrac{d \mu_k^\ep(x)\, d \mu_k^\ep(y) }{|(x-y) \cdot e_\theta|^{s-\ep} \vee 1}\right] d\theta \\  
	 I_2	& = \int_0^\pi \left[ \sum_{n\geq 0} 2^{n(s-\ep)} \sum_{ k> j \geq n} \int_{E \cap S_j \cap C_n^\theta} \int_{E \cap S_k \cap C_n^\theta}  \dfrac{d \mu_k^\ep(x)\, d \mu_j^\ep(y) }{|(x-y) \cdot  e_\theta|^{s-\ep} \vee 1}\right] d\theta.
 \end{align*}
 Starting with $I_1$, one has  
 \begin{align*}
 	I_1= 
 	& \int_0^\pi \left[ \sum_{n\geq 0}  2^{n(s- \ep)}  \sum_{k\geq n}   \iint_{(E \cap S_k \cap C_n^\theta)^2}    \dfrac{d \mu_k^\ep(x)\, d \mu_k^\ep(y) }{|(x-y) \cdot e_\theta|^{s-\ep} \vee 1}\right] d\theta \\ =
 	& \int_0^\pi \left[ \sum_{n\geq 0}  2^{n(s- \ep)}  \sum_{k\geq n}  \iint_{(E \cap S_k  )^2}   \dfrac{ \mathds{1}_{C_n^\theta}(x) \mathds{1}_{ C_n^\theta}(y) }{|(x-y) \cdot e_ \theta|^{s-\ep} \vee 1}  d \mu_k^\ep(x)\, d \mu_k^\ep(y)\right] d\theta \\ =
 	& \sum_{n\geq 0}  2^{n(s- \ep)}  \sum_{k\geq n}    \iint_{(E \cap S_k  )^2}   \int_0^\pi \dfrac{ \mathds{1}_{C_n^\theta}(x) \mathds{1}_{ C_n^\theta}(y)  }{|(x-y) \cdot e_ \theta|^{s-\ep} \vee 1} d\theta d \mu_k^\ep(x)\, d \mu_k(y)  \\ 
	 \leq 
 	& \sum_{n\geq 0}  2^{n(s-\ep)}  \sum_{k\geq n}    \iint_{(E \cap S_k  )^2}   \left[\int_0^\pi \dfrac{ \mathds{1}_{x  \in C_n^\theta} (\theta)\mathds{1}_{y \in C_n^\theta}(\theta) }{\lvert \tau_{x-y} \cdot  e_\theta \rvert^{s-\ep} }  d\theta\right] \dfrac{d \mu_k^\ep(x)\, d \mu_k^\ep(y) }{\norm{x-y}^{s-\ep} \vee 1},
 \end{align*}
 where $\tau_{x-y}$ is the unit vector in the direction of $x-y$.  By Lemma \ref{lemma:n,k}, when $x\in  S_k$ one has $\mathds{1}_{x  \in C_n^\theta}(\theta) = \mathds{1}_{J_{n,x} }(\theta) $. Then
 \begin{align*}
 	\int_0^\pi \dfrac{\mathds{1}_{x  \in C_n^\theta}(\theta) \mathds{1}_{y \in C_n^\theta}(\theta)}{\lvert \tau_{x-y} \cdot  e_\theta \rvert^{s-\ep} }  d\theta = 
 	& \int_{J_{n,x}  \cap J_{n,y} } \dfrac{d\theta }{\lvert \cos(\widehat{\tau_{x-y}, e_\theta}) \rvert^{s-\ep} }.
 \end{align*}
 By  Lemma \ref{lemma:n,k},  the interval $J_{n,x}  \cap J_{n,y} $ has length smaller than $C_0 2 ^{n-k}$. So the integral above is taken over an   interval of length at most $C_0 2 ^{n-k}$. Moreover, as $s<1$, the integral reaches its largest value when  $\theta$ close to $\dfrac{\pi}{2}$. Thus 
 \begin{align}
 	\label{eq:int_t}
 	\int_0^\pi \dfrac{\mathds{1}_{x  \in C_n^\theta}(\theta) \mathds{1}_{y \in C_n^\theta}(\theta)}{\lvert \tau_{x-y} \cdot e_ \theta \rvert^{s-\ep} }  d\theta \leq  
 	& \int_{\frac{\pi}{2}- C_0 2 ^{n-k} }^{\frac{\pi}{2}+ C_0 2 ^{n-k} } \dfrac{d\theta }{\lvert \cos(\theta) \rvert^{s-\ep} }
	 \leq  \int_{- C_0 2 ^{n-k} } ^{C_0 2 ^{n-k}}   \dfrac{d\theta }{\lvert \theta \rvert^{s-\ep} }= C  2^{(n-k)(1-s+\ep)}.
 \end{align}
 where $C>0$ is some positive constant. Then going back to $I_1$ and using \refeq{eq:int_t}, one gets
 \begin{align*}
 	I_1
 	&  \leq  C \sum_{n\geq 0}  2^{n(s-  \ep)}  \sum_{k\geq n} 2^{ (n-k)(1-s+\ep)} \iint_{(E \cap S_k)^2} \dfrac{d \mu_k^\ep(x)\, d \mu_k^\ep(y) }{\norm{x-y}^{\ep} \vee 1} \\
 	& = C \sum_{n\geq 0} \sum_{k \geq n} 2^{n+k(s+ \ep -1)}  \iint_{(E \cap S_k)^2} \dfrac{d \mu_k^\ep(x)\, d \mu_k^\ep(y) }{\norm{x-y}^{s-\ep} \vee 1 } \\
 	&  = C \sum_{n\geq 0}2^{n(s+\ep -1)} \sum_{k=0}^n 2^{k}  \iint_{(E \cap S_n)^2} \dfrac{d \mu_n^\ep(x)\, d \mu_n^\ep(y) }{\norm{x-y}^{s-\ep}  \vee 1 } \\ 
 	& \leq 2 C \sum_{n\geq 0} 2^{n(s-\ep)}  I_{s-\ep}(\mu_n^\ep)  <+\infty,
 \end{align*}
which is finite by \eqref{eq12}.

\medskip

Moving to $I_2$, the same manipulations as above for $I_1$ yield 
 \begin{align*}
 	I_2
	& = \int_0^\pi \left[ \sum_{n\geq 0} 2^{n(s- \ep)} \sum_{ k> j \geq n} \int_{E \cap S_j \cap C_n^\theta} \int_{E \cap S_k \cap C_n^\theta}  \dfrac{d \mu_k^\ep(x)\, d \mu_j^\ep(y) }{|(x-y) \cdot e_ \theta|^{s-\ep} \vee 1}\right] d\theta.\\
& =  \int_0^\pi \left[ \sum_{n\geq 0}  2^{n(s- \ep)}   \sum_{ k> j \geq n}  \int_{E \cap S_j} \int_{E \cap S_k} \dfrac{ \mathds{1}_{C_n^\theta}(x) \mathds{1}_{C_n^\theta}(y) }{|(x-y) \cdot  e_\theta|^{s-\ep} }  d \mu_k^\ep(x)\, d \mu_j^\ep(y)\right] d\theta \\	& = 
 	 \sum_{n\geq 0}  2^{n(s- \ep)}  \sum_{ k> j \geq n}   \int_{E \cap S_j} \int_{ E \cap S_k}  \left[\int_0^\pi \dfrac{ \mathds{1}_{J_{n,x} }(\theta)   \mathds{1}_{J_{n, y} }(\theta)     }{\lvert \tau_{x-y} \cdot  e_\theta \rvert^{s-\ep} }  d\theta\right] \dfrac{d \mu_k^\ep(x)\, d \mu_j^\ep(y) }{\left\lVert x-y \right\rVert_2 ^{s-\ep} } . 
 	 \end{align*}
As before, by Lemma \ref{lemma:n,k}, $\lvert  J_{k,x} \rvert \leq 2^{n-k}$  and $\lvert  J_{j,y} \rvert \leq 2^{n-j}$ for all $x \in S_k\cap C_n^\theta $  and $ y \in S_j \cap C_n^\theta$). Then, as $k \geq j+1$, 	the same argument as in \eqref{eq:int_t} yields
	\begin{align}
		\label{eq:intkj}
		\int_0^\pi \dfrac{\mathds{1}_{x  \in C_n^\theta}(\theta) \mathds{1}_{y \in C_n^\theta}(\theta)}{\lvert \tau_{x-y} \cdot  e_\theta \rvert^{s-\ep} }  d\theta \leq C 2^{(n-k)(1-s+\ep)}.
	\end{align}
	for some $C >0$.

Next, we make use of equation \eqref{eq-separation} : indeed, it is not possible that $\mu_j^\ep$ and $\mu_{j+1}^\ep$ are simultaneously non-zero. Hence,  for $x \in S_k$ and $y \in S_j$ such that $j <k$ and $\mu_j^\ep$ and $\mu_{k}^\ep$ not both equal to zero, then necessarly $|k-j|\geq 2$ and  $2^{k-2}\leq  \left\lVert x-y \right\rVert_2 \leq 2^{k+1}$. This implies in particular that
	 \begin{align}
	 	\label{eq:mu_kmu_j}
	 	\int_{E \cap S_j} \int_{ E \cap S_k} \dfrac{d \mu_k^\ep(x)\, d \mu^\ep_j(y) }{\left\lVert x-y \right\rVert_2 ^{s-\ep} } \leq C 2^{-k(s-\ep)}  \mu_k^\ep(E \cap S_k) \, \mu_j^\ep(E \cap S_j),
	 \end{align}
	 the inequality being in fact close to be sharp. 
	
	Finally, combining \eqref{eq:mu_kmu_j} and \eqref{eq:intkj}), one gets that  for some $C'>0$, 
 	\begin{align*}
 		I_2 &  \leq 
 		  C'\sum_{n\geq 0}  2^{n(s- \ep)}  \sum_{ k> j \geq n}  2^{(n-k)(1-s+\ep)} 2^{-k(s-\ep)} \mu_k^\ep(E \cap S_k) \, \mu_j^\ep(E \cap S_j)\\
		  & = 
 		 C'\sum_{n\geq 0}  2^{n }  \sum_{ k> j \geq n}   2^{-k}   \mu_j^\ep(E \cap S_j) \mu_k^\ep(E \cap S_k) \\ 
		 &=  C'\sum_{j\geq 0}  \left(\sum_{n=0}^j 2^{n }\right)   \mu_n^\ep(E \cap S_n) \sum_{k\geq n+1} 2^{-k} \mu_k^\ep(E \cap S_k) \\ 
		 &\leq  
 		 C'\sum_{n\geq 0} 2^{n} \mu_n^\ep(E \cap S_n) \sum_{k\geq n+1} 2^{-k} \mu_k^\ep(E \cap S_k) \\ 
		 & \leq  C'\sum_{n\geq 0} 2^{-n}  \mu_n^\ep(E \cap S_n)\left (\sum_{k=0}^n 2^{k }\mu_k^\ep(E \cap S_k) \right) .
 	\end{align*}
This last double sum is finite, because the set $E$ was chosen so that \eqref{eq-finale2} holds true. This concludes the proof.	   \end{proof}

% 
%%%%%%%%%%%%%%%%%%%%%%%%%%%%%%%%% 
%%%%%%%%%%%%%%%%%%%%%%%%%%%%%%%%% 
%%%%%%%%%%%%%%%%%%%%%%%%%%%%%%%%% 
%%%%%%%%%%%%%%%%%%%%%%%%%%%%%%%%% 
%%%%%%%%%%%%%%%%%%%%%%%%%%%%%%%%% 
%%%%%%%%%%%%%%%%%%%%%%%%%%%%%%%%% 
 
%%%%%%%%%%%%%%%%%%%%%%%%%%%%%%%%%%%%%%%%%%%%%%%%%%%%%%%
%%%%%%%%%%%%%%%%%%%% BIBLIOGRAPHY %%%%%%%%%%%%%%%%%%%%%
%%%%%%%%%%%%%%%%%%%%%%%%%%%%%%%%%%%%%%%%%%%%%%%%%%%%%%%

\bibliography{biblio_lara} 

\begin{thebibliography}{10}

\bibitem{barlow1989}
M.~T. Barlow and S.~J. Taylor.
\newblock Fractional dimension of sets in discrete spaces.
\newblock {\em J. Phys. A}, 22(13):2621--2628, 1989.
\newblock With a reply by J. Naudts.

\bibitem{barlow1992}
M.~T. Barlow and S.~J. Taylor.
\newblock Defining fractal subsets of {${\bf Z}\sp d$}.
\newblock {\em Proc. London Math. Soc. (3)}, 64(1):125--152, 1992.

\bibitem{daw2021uniform}
L.~Daw.
\newblock A uniform result for the dimension of fractional {B}rownian motion
  level sets.
\newblock {\em Statistics \& Probability Letters}, 169:108984, 2021.

\bibitem{daw2021fractal}
L.~Daw and G.~Kerchev.
\newblock Fractal dimensions of the {R}osenblatt process.
\newblock {\em arXiv:2103.04714v1}, 2021.

\bibitem{evans2015measure}
L.~C. Evans and R.~F Gariepy.
\newblock {\em Measure theory and fine properties of functions}.
\newblock CRC press, 2015.

\bibitem{falconer2003}
K.~Falconer.
\newblock {\em Fractal geometry}.
\newblock John Wiley \& Sons, Inc., Hoboken, NJ, second edition, 2003.
\newblock Mathematical foundations and applications.

\bibitem{falconer2004fractal}
K.~Falconer.
\newblock {\em Fractal geometry: {M}athematical foundations and applications}.
\newblock John Wiley \& Sons, 2004.

\bibitem{georgiou2018dimension}
N.~Georgiou, D.~Khoshnevisan, K.~Kim, and A.~D. Ramos.
\newblock The dimension of the range of a transient random walk.
\newblock {\em Electronic Journal of Probability}, 23:1--31, 2018.

\bibitem{hochman2012local}
M.~Hochman and P.~Shmerkin.
\newblock Local entropy averages and projections of fractal measures.
\newblock {\em Annals of Mathematics}, pages 1001--1059, 2012.

\bibitem{kaufman1968hausdorff}
R.~Kaufman.
\newblock On {H}ausdorff dimension of projections.
\newblock {\em Mathematika}, 15(2):153--155, 1968.

\bibitem{khoshnevisan2017intermittency}
D.~Khoshnevisan, J.~Kim, and Y.~Xiao.
\newblock Intermittency and multifractality: A case study via parabolic
  stochastic {PDE}s.
\newblock {\em Annals {P}robab}, 45(6A):3697--3751, 2017.

\bibitem{khoshnevisan2017macroscopic}
D.~Khoshnevisan and Y.~Xiao.
\newblock On the macroscopic fractal geometry of some random sets.
\newblock In {\em Stochastic Analysis and Related Topics}, pages 179--206.
  Springer, 2017.

\bibitem{marstrand1954some}
J.~M. Marstrand.
\newblock Some fundamental geometrical properties of plane sets of fractional
  dimensions.
\newblock {\em Proceedings of the London Mathematical Society}, 3(1):257--302,
  1954.

\bibitem{mattila1975hausdorff}
P.~Mattila.
\newblock Hausdorff dimension, orthogonal projections and intersections with
  planes.
\newblock {\em Ann. Acad. Sci. Fenn. Ser. AI Math}, 1(2):227--244, 1975.

\bibitem{nourdin2018sojourn}
I.~Nourdin, G.~Peccati, and S.~Seuret.
\newblock Sojourn time dimensions of fractional {B}rownian motion.
\newblock {\em Bernoulli 26(3): 1619-1634}, 2020.

\bibitem{seuret2019sojourn}
S.~Seuret and X.~Yang.
\newblock On sojourn of {B}rownian motion inside moving boundaries.
\newblock {\em Stoch. {P}rocesses {A}pp.}, 129(3):978--994, 2019.

\bibitem{xiao2013}
Y.~Xiao and X.~Zheng.
\newblock Discrete fractal dimensions of the ranges of random walks in {${\Bbb
  Z}\sp d$} associate with random conductances.
\newblock {\em Probab. Theory Related Fields}, 156(1-2):1--26, 2013.

\end{thebibliography}
\bibliographystyle{plain}

\end{document}